\documentclass[AMS,STIX1COL]{WileyNJD-v2}
\allowdisplaybreaks[4]
\usepackage{graphicx}
\usepackage{epstopdf}
\articletype{Research Article}%

\begin{document}

\title{Decentralized Control for Discrete-time Mean-Field Systems with Multiple Controllers of Delayed Information}

\author[1]{Qingyuan Qi}

\author[2]{Zhiqiang Liu}

\author[2]{Qianqian Zhang}

\author[1]{Xinbei Lv*}
\authormark{QI \textsc{et al.}}

\address[1]{\orgdiv{Qingdao Innovation and Development Center}, \orgname{Harbin Engineering University}, \orgaddress{\state{Qingdao}, \country{China}}}

\address[2]{\orgdiv{Institute of Complexity Science, School of Automation}, \orgname{Qingdao University}, \orgaddress{\state{Qingdao}, \country{China}}}

\corres{Xinbei Lv, Qingdao Innovation and Development Center, Harbin Engineering University, Qingdao, China 266000.\\
\email{lvxinbei@hrbeu.edu.cn}}

\abstract[Abstract]{In this paper, the finite horizon asymmetric information linear quadratic (LQ) control problem is investigated for a discrete-time mean field system. Different from previous works, multiple controllers with different information sets are involved in the mean field system dynamics. The coupling of different controllers makes it quite difficult in finding the optimal control strategy. Fortunately, by applying the Pontryagin's maximum principle, the corresponding decentralized control problem of the finite horizon is investigated. The contributions of this paper can be concluded as: For the first time, based on the solution of a group of mean-field forward and backward stochastic difference equations (MF-FBSDEs), the necessary and sufficient solvability conditions are derived for the asymmetric information LQ control for the mean field system with multiple controllers. Furthermore, by the use of an innovative orthogonal decomposition approach, the optimal decentralized control strategy is derived, which is based on the solution to a non-symmetric Riccati-type equation.}


\maketitle


\section{Introduction}
Different from the controlled linear stochastic differential/difference equations (SDEs) studied in classical stochastic control problem, the mathematical expectation terms appear in the mean-field SDEs. Due to the applications in large population stochastic dynamics games, and various physical and sociological dynamical systems, the study of mean-field SDEs can be traced back to 1950s, and abundant research results have been obtained, see \citep{k1956}-\citep{qxz2021}. Particularly, the continuous time mean-field LQ control problem was firstly studied in \citep{y2013}, and a necessary and sufficient solvability condition was proposed. Furthermore, both the finite horizon LQ control problem and the infinite horizon stabilization control problem for discrete time mean-field systems were solved in \citep{zq2016, zqf2019}. Moreover, the indefinite LQ control for mean-field systems is investigated in \citep{wz2021}, \citep{lly2020}.

It is important to note that the previously mentioned studies on mean-field control problem are mainly symmetric information control problem, only a centralized control strategy needs to be designed. As is well known, the traditional centralized control problem with one single controller were well studied since the last century. In addition, for the system dynamics with multiple controllers of the same information structure, the corresponding decentralized control problems can be converted into a centralized control problem by using the system augmentation approach. Decentralized control systems have multiple controllers that are collaboratively trying to control a system by taking actions based on their individual observations. The observations of one controller may not be available to the other controllers. Whereas, in contrast to centralized control systems, the decentralized control problem with controllers of asymmetric information structure remains less investigated.

We should emphasis that real world systems usually contain multiple controllers, and each controller accesses its individual information. For this situation, finding optimal control strategy is usually difficult in view of the coupling of different controllers. The pioneering study of asymmetric information control is the well-celebrated Witsenhausen's counterexample raised in 1968 (see \citep{wh1968}), which shows that the linear control strategy is no longer optimal for a decentralized control of linear dynamical system, and the associated optimal control problem remains unsolved. Since then, in view of the wide applications in many fields, the research of decentralized control with asymmetric information controllers has attracted much interest from researchers in recent years. For example, the stochastic game problem with asymmetric information was investigated and solved in \citep{nglb2013,gnlb2014}. The optimal local and remote decentralized control problem was investigated via the common information approach, see \citep{lx2018, aon2019}. Besides, the decentralized control for linear stochastic system with multiple controllers of different information structures was studied in \citep{lz2016}-\citep{qxz2020}.

Different from the previous works mentioned above, we will investigate a special kind of decentralized control problem for discrete-time mean-field systems. Specifically, multiple controllers with delayed information patterns as well as the mean-field terms (i.e., the expectations of the controllers and the state) are contained within the system dynamics. Meanwhile, each controller can access its individual information, which is different with each other. Our goal is to design the optimal decentralized control strategy to minimize a given quadratic cost function. It should be pointed out that the studied problem has not been solved so far. The existence of mean-field terms and the delayed asymmetric information structure make the studied problem challenging in the following aspects: 1) In view of the asymmetric information structure, the controllers are coupled with each other, hence finding the optimal decentralized control is difficult, see \citep{lx2018, aon2019}; 2) Due to the existence of the mean-filed terms, the original optimal decentralized control problem cannot be considered as a standard LQ control problem, and the explicit solvability conditions have not been derived, see \citep{y2013}.

In this paper, by applying the maximum principle, the corresponding decentralized control problem for mean-field systems with controllers of delayed information structures is well studied. In the first place, by the use of the variational method, the necessary and sufficient solvability conditions are given in accordance with a group of MF-FBSDEs. Subsequently, in order to decouple the associated MF-FBSDEs, an innovative orthogonal decomposition approach is proposed. Hence, in view of the delayed information structure, it is shown that the MF-FBSDEs are decoupled, and the relationship between the costate and the state is established. Finally, the optimal decentralized control strategy is derived by introducing asymmetric Riccati equations.

Very recently, the optimal LQ control problem with multiple controllers of non-symmetric information structure was  investigated in \citep{qxz2020}.  {\color{blue}It is worth noting that the adopted methodology exhibits innovation in the following aspects: Firstly, to deal with the mean field terms in mean-field system, a group of mean-field forward and backward stochastic difference equations (MF-FBSDEs) is firstly introduced. Obviously, the MF-FBSDEs were not mentioned in reference \citep{qxz2020}, and we show that the optimal decentralized control can be uniquely solved if and only if the MF-FBSDEs can be uniquely decoupled. Secondly, as a result, in order to obtain the optimal decentralized control for mean-field systems, we introduce a novel technique for decoupling the MF-FBSDEs. Specifically, we derive the control mathematical expectation first, followed by obtaining the optimal decentralized control. It is worth noting that this approach has not been utilized in reference \citep{qxz2020}. Furthermore, this paper demonstrates that the optimal decentralized control can be computed by solving a set of new asymmetric Riccati equations.}

{\color{blue}The contributions of the paper are: 1)   The necessary and sufficient conditions for the discrete-time mean field decentralized control problem with multiple controllers of delayed information patterns are derived. 2) Utilizing a novel orthogonal decomposition approach, we have successfully obtained the solution to the group of MF-FBSDEs, which, to the best of our knowledge, represents a new and original contribution. 3) With regards to the optimal predictor, we have successfully overcome the challenges associated with the coupling between multiple controllers, enabling us to derive the optimal decentralized control for the first time.}

The necessary and sufficient conditions for the discrete-time mean field decentralized control problem with multiple controllers of delayed information patterns are derived. 2) According to a novel orthogonal decomposition approach, the solution to the group of MF-FBSDEs is obtained, which is new as far as our knowledge. 3) In terms of the optimal predictor, the barrier of the coupling between the multiple controllers are overcome, the optimal decentralized control are thus derived for the first time.

The structure of the paper is as follows. In Section 2, the problem under consideration is formulated, while the solvability conditions of the problem is shown in Section 3. The optimal decentralized control strategy is developed in Section 4. A numerical example is given in Section 5, and we conclude the paper in Section 6. Finally, there are some relevant detailed proofs in Appendix.

\textbf{Notations}: $I_{n}$ means the unit matrix with rank $n$; Superscript $\mathcal{A}^{T}$ denotes the transpose of a matrix. Real symmetric matrix $\mathcal{A}>0 \ (\text{or}\geq0)$ implies that $\mathcal{A}$ is strictly positive definite (or positive semi-definite). $\mathbb{R}^{n}$ signifies the $n$-dimensional Euclidean space. $\mathcal{B}^{-1}$ represents the inverse of real matrix $\mathcal{B}$. Denote the natural filtration $\mathcal{F}_{i}(\tau)$ as an the $\sigma$-algebra generated by $\{x(0),\omega(0),\cdots,\omega(\tau-i-1)\}$ and augmented by all the $\mathcal{P}$-null sets. $\mathbb{E}[\cdot|\mathcal{F}_{i}(\tau)]$ means the conditional expectation with respect to $\mathcal{F}_{i}(\tau)$.
For the convenience of presentation, the following concise notations are introduced.
\begin{equation}
\mathbf{R}=\begin{bmatrix}
 \mathcal{R}_{0}& 0 & \cdots  & 0\\
 0&  \mathcal{R}_{1}& \cdots  & 0\\
 \vdots & \vdots  & \ddots  & \vdots \\
 0& 0 & \cdots  &\mathcal{R}_{h}
\end{bmatrix},
\mathbf{\bar{R}}=\begin{bmatrix}
 \bar{\mathcal{R}}_{0}& 0 & \cdots  & 0\\
 0&  \bar{\mathcal{R}}_{1}& \cdots  & 0\\
 \vdots & \vdots  & \ddots  & \vdots \\
 0& 0 & \cdots  &\bar{\mathcal{R}}_{h}
\end{bmatrix},\notag\\
\end{equation}
\begin{equation}
\mathbf{B}=\begin{bmatrix}
\mathcal{B}{_{0}}^{T}\\
\mathcal{B}{_{1}}^{T}\\
\vdots \\
\mathcal{B}{_{h}}^{T}
\end{bmatrix}^{T},
\mathbf{\bar{B}}=\begin{bmatrix}
\bar{\mathcal{B}}{_{0}}^{T}\\
\bar{\mathcal{B}}{_{1}}^{T}\\
\vdots \\
\bar{\mathcal{B}}{_{h}}^{T}
\end{bmatrix}^{T},
\mathbf{D}=\begin{bmatrix}
D{_{0}}^{T}\\
D{_{1}}^{T}\\
\vdots \\
D{_{h}}^{T}
\end{bmatrix}^{T},
\mathbf{\bar{D}}=\begin{bmatrix}
\bar{D}{_{0}}^{T}\\
\bar{D}{_{1}}^{T}\\
\vdots \\
\bar{D}{_{h}}^{T}
\end{bmatrix}^{T},\notag\\
\end{equation}
\begin{equation}
\mathbf{V}(\tau)=\begin{bmatrix}
v{_{0}}(\tau)\\
v{_{1}}(\tau)\\
\vdots \\
v{_{h}}(\tau)
\end{bmatrix},
\Delta \mathbf{V}(\tau)=\begin{bmatrix}
\Delta v{_{0}}(\tau)\\
\Delta v{_{1}}(\tau)\\
\vdots \\
\Delta v{_{h}}(\tau)
\end{bmatrix},
\mathbb{E}\mathbf{V}(\tau)=\begin{bmatrix}
\mathbb{E}v{_{0}}(\tau)\\
\mathbb{E}v{_{1}}(\tau)\\
\vdots \\
\mathbb{E}v{_{h}}(\tau)
\end{bmatrix},
\mathbb{E}\Delta \mathbf{V}(\tau)=\begin{bmatrix}
\mathbb{E}\Delta v{_{0}}(\tau)\\
\mathbb{E}\Delta v{_{1}}(\tau)\\
\vdots \\
\mathbb{E}\Delta v{_{h}}(\tau)
\end{bmatrix},\notag\\
\end{equation}
\begin{equation}
\mathcal{A}(\tau)=\mathcal{A}+\omega (\tau)C,\bar{\mathcal{A}}(\tau)=\bar{\mathcal{A}}+\omega (\tau)\bar{C},\notag\\
\end{equation}
\begin{equation}
\mathcal{B}_{i}(\tau)=\mathcal{B}_{i}+\omega (\tau)D_{i},\bar{\mathcal{B}}_{i}(\tau)=\bar{\mathcal{B}}_{i}+\omega (\tau)\bar{D}_{i},\notag\\
\end{equation}
\begin{equation}
\mathbf{B}(\tau)=\mathbf{B}+\omega (\tau)\mathbf{D},\bar{\mathbf{B}}(\tau)=\bar{\mathbf{B}}+\omega (\tau)\bar{\mathbf{D}},\notag\\
\end{equation}

\begin{equation}
\mathbf{R}_{i}=\begin{bmatrix}
 \mathcal{R}_{0}& 0 & \cdots  & 0\\
 0& \mathcal{R}_{1} & \cdots  & 0\\
\vdots  & \vdots  & \ddots  & \vdots \\
 0& 0 & \cdots  & \mathcal{R}_{i}
\end{bmatrix},
\mathbf{\bar{R}}_{i}=\begin{bmatrix}
\bar{\mathcal{R}}_{0}& 0 & \cdots  & 0\\
 0& \bar{\mathcal{R}}_{0} & \cdots  & 0\\
\vdots  & \vdots  & \ddots  & \vdots \\
 0& 0 & \cdots  & \bar{\mathcal{R}}_{i}
\end{bmatrix},\notag\\
\end{equation}
\begin{equation}\label{equ:1}
\mathbf{V}_{i}(\tau)=\begin{bmatrix}
\hat{v}_{0,i}(\tau)\\
\vdots \\
\hat{v}_{i-1,i}(\tau)\\
v_{i}(\tau)
\end{bmatrix},
\mathbb{E}\mathbf{V}_{i}(\tau)=\begin{bmatrix}
\mathbb{E}\hat{v}_{0,i}(\tau)\\
\vdots \\
\mathbb{E}\hat{v}_{i-1,i}(\tau)\\
\mathbb{E}v_{i}(\tau)
\end{bmatrix},
\end{equation}
\begin{equation}
\mathbf{B}_{i}=\begin{bmatrix}
\mathcal{B}_{0}^{T}\\
\mathcal{B}_{1}^{T}\\
\vdots \\
\mathcal{B}_{i}^{T}
\end{bmatrix},
\mathbf{\bar{B}}_{i}=\begin{bmatrix}
\bar{\mathcal{B}}_{0}^{T}\\
\bar{\mathcal{B}}_{1}^{T}\\
\vdots \\
\bar{\mathcal{B}}_{i}^{T}
\end{bmatrix},
\mathbf{D}_{i}=\begin{bmatrix}
D_{0}^{T}\\
D_{1}^{T}\\
\vdots \\
D_{i}^{T}
\end{bmatrix},
\mathbf{\bar{D}}_{i}=\begin{bmatrix}
\bar{D}_{0}^{T}\\
\bar{D}_{1}^{T}\\
\vdots \\
\bar{D}_{i}^{T}
\end{bmatrix},\notag\\
\end{equation}
\begin{equation}
\mathbf{B}_{i}(\tau)=\mathbf{B}_{i}+\omega(\tau)\mathbf{D}_{i},
\mathbf{\bar{B}}_{i}(\tau)=\mathbf{\bar{B}}_{i}+\omega(\tau)\mathbf{\bar{D}}_{i}, \text{for}\ i=0,\cdots,h.\notag
\end{equation}

\section{Problem Formulation}
We consider the following discrete-time mean-field stochastic system with multiple controllers
\begin{align}\label{equ:2}
x(\tau+1)=\Big[\mathcal{A}x(\tau)+\bar{\mathcal{A}}\mathbb{E}x(\tau)+\sum_{i=0}^{h}[\mathcal{B}_{i}v_{i}(\tau)+\bar{\mathcal{B}}_{i}\mathbb{E}v_{i}(\tau)]\Big]+\Big[Cx(\tau)+\bar{C}\mathbb{E}x(\tau)+\sum_{i=0}^{h}[D_{i}v_{i}(\tau)+\bar{D}_{i}\mathbb{E}v_{i}(\tau)]\Big]\omega (\tau),
\end{align}
where $\tau$ is the time instant, $x(\tau)\in{\mathbb{R}^n}$ is the system state, $v_{i}(\tau)\in \mathbb{R}^{m_{i}}$ is the $i$-th control input. $\mathcal{A},\bar{\mathcal{A}},C,\bar{C}\in \mathbb{R}^{n\times n}$, and $\mathcal{B}_{i},\bar{\mathcal{B}}_{i},D_{i},\bar{D}_{i}\in \mathbb{R}^{n\times m}, 0\leq i\leq h$ are matrices of appropriate dimensions, $\omega(\tau)$ is a scalar-valued Gaussian white noise with $\omega(\tau)\sim \mathcal{N}(0, \sigma^{2})$. The initial state $x(0)=\beta$ is given, and $\mathbb{E}$ is the expectation taken over the noise $\omega(\tau)$ and initial state $\beta$.

To guarantee the solvability of system \eqref{equ:2}, the initial control strategies $v_{i}(\tau) (0\leq i\leq h,0\leq \tau\leq i-1)$ are given arbitrarily.

It is noted that the expectations $\mathbb{E}x(\tau), \mathbb{E}v_{i}(\tau)$ are involved in the system dynamics \eqref{equ:2}, which will cause essential difficulties in deriving the optimal decentralized control strategy. For the sake of discussion, it can be derived from \eqref{equ:2} that
\begin{align}\label{equ:3}
\mathbb{E}x(\tau+1)=(\mathcal{A}+\bar{\mathcal{A}})\mathbb{E}x(\tau)+\sum_{i=0}^{h}(\mathcal{B}_{i}+\bar{\mathcal{B}}_{i})\mathbb{E}v_{i}(\tau).
\end{align}

Throughout this paper, for (\ref{equ:2})-(\ref{equ:3}), the following assumption on the information structure is made:
\begin{assumption}\label{asss}
The $i$-th controller $v_{i}(\tau)$ is $\mathcal{F}_{i}(\tau)$-measurable,
in which $\mathcal{F}_{i}(\tau)$ is subject to
$\mathcal{F}_{i}(\tau)=\sigma\{\beta,\omega(0),\cdots,\omega(\tau-i-1)\}.$
\end{assumption}

Clearly, for $0\leq i\leq h$ and $i\leq \tau\leq \Gamma$, {we can infer}
$\mathcal{F}_{0}(\tau)\supset \mathcal{F}_{1}(\tau)\supset \cdots \supset \mathcal{F}_{h}(\tau),h\leq \tau, \text{and}\ \mathcal{F}_{i}(\Gamma)\supset \mathcal{F}_{i}(\Gamma-1)\supset \cdots \supset\mathcal{F}_{i}(\tau), i\leq \tau.
$
\begin{remark}\label{remark1}
For mean-field system \eqref{equ:2}, $h+1$ controllers are involved, and each controller can access its individual information, which is different with each other. Such kind of system is called asymmetric information stochastic system, and the corresponding optimal control problem turns into a decentralized control problem, and finding the optimal decentralized control strategy is usually difficult, see \citep{lx2018,aon2019,qxz2020,wfll2022, wh2021, nglb2013}.
\end{remark}

Corresponding with  system (\ref{equ:2})-\eqref{equ:3}, {the cost functional is given by}
\begin{align}\label{equ:4}
 J_{\Gamma}(v)=&\sum_{\tau=0}^{\Gamma}\mathbb{E}\Big[x^T(\tau)Qx(\tau)+[\mathbb{E}x(\tau)]^T\bar{Q}\mathbb{E}x(\tau)+\sum_{i=0}^{h}[v_{i}^T(\tau)\mathcal{R}_{i}v_{i}(\tau)+[\mathbb{E}v_{i}(\tau)]^T\bar{\mathcal{R}}_{i}\mathbb{E}v_{i}(\tau)]\Big]\notag\\
&+\mathbb{E}[x^T(\Gamma+1)\Phi(\Gamma+1)x(\Gamma+1)]+[\mathbb{E}x(\Gamma+1)]^T\bar{\Phi}(\Gamma+1)\mathbb{E}x(\Gamma+1),
\end{align}
where $Q,\bar{Q},\mathcal{R}_{i},\bar{\mathcal{R}}_{i},\Phi(\Gamma+1),\bar{\Phi}(\Gamma+1)$ are deterministic symmetric weighting matrices with appropriate dimensions.

For the weighting matrices in \eqref{equ:4}, we might as well assume:
\begin{assumption}\label{assumption:1}
$Q\geq0,Q+\bar{Q}\geq0,\mathcal{R}_{i}>0,\mathcal{R}_{i}+\bar{\mathcal{R}}_{i}>0,\Phi(\Gamma+1)\geq0,\Phi(\Gamma+1)+\bar{\Phi}(\Gamma+1)\geq0.$
\end{assumption}

{\color{blue}{ In what follows, we will present an introduction to the decentralized control problem of mean-field system with multiple controllers of delayed information.}}
\begin{problem}\label{problem:1}
Find $\mathcal{F}_{i}(\tau)$-measurable controllers
$v_{i}(\tau)\in \mathbb{R}^{m_{i}}, i=0,\cdots,h$ to minimize \eqref{equ:4}.
\end{problem}

\begin{remark}\label{remark2}
The solution to Problem \ref{problem:1} is hard to obtain, and the optimal decentralized control strategy has not been derived before. The reasons are twofold: 1) The information set available to each controller is different, which results in the coupling of $h+1$ controllers, and makes the finding of optimal decentralized control strategy difficult. 2) The mathematical expectation terms $\mathbb{E}x(\tau), \mathbb{E}v_{i}(\tau)$ appear in \eqref{equ:2}-\eqref{equ:4}, which destroys the adaptability of the control inputs, and consequently, Problem \ref{problem:1} cannot be solved by applying traditional methods such as system augmentation.
\end{remark}

\section{Existence of Optimal Decentralized Control Strategy}
We are going to consider the existence of the solution to Problem \ref{problem:1} via the variational method. To begin with, the following lemma shall be given, which serves as the preliminary.
\begin{lemma}\label{lemma:1}
For (\ref{equ:2}) and (\ref{equ:4}), set $\varepsilon\in\mathbb{R}$, and denote $v_{i}^{\varepsilon}(\tau)=v_{i}(\tau)+\varepsilon\Delta v_{i}(\tau)$, $0\leq i\leq h,i\leq \tau\leq \Gamma$, in which $\Delta v_{i}(\tau)$ is $\mathcal{F}_i(\tau)$-adapted with $\sum_{\tau=0}^{\Gamma}\Delta {v_{i}}^{T}(\tau)\Delta v_{i}(\tau)< +\infty$. We have
\begin{align}\label{equ:8}
J_{\Gamma}(v^\varepsilon)&- J_{\Gamma}(v)=\varepsilon^{2}\Delta J_{\Gamma}(\Delta v)\notag\\
&+2\varepsilon\sum_{t=0}^{\Gamma}\mathbb{E}\Big[\big[\mathbf{B}^{T}(\tau)\theta(\tau)+\mathbb{E}[\mathbf{\bar{B}}^{T}(\tau)\theta(\tau)]+\mathbf{R}\mathbf{V}(\tau)+\mathbf{\bar{R}}\mathbb{E}\mathbf{V}(\tau)\big]^{T}\Delta \mathbf{V}(\tau)\Big].
\end{align}
In the above, $x^{\varepsilon}(\tau)$ and $J_{\Gamma}(v^{\varepsilon})$ are the state variable and the cost functional corresponding with $v_{i}^{\varepsilon}(\tau)$, respectively. Moreover, $\Delta J_{\Gamma}(\Delta v)$ can be calculated as
\begin{align}\label{equ:9}
 \Delta J_{\Gamma}(\Delta v)&=\sum_{\tau=0}^{\Gamma}\mathbb{E}\big[\eta^{T}(\tau)Q\eta(\tau)+[\mathbb{E}\eta(\tau)]^{T}\bar{Q}\mathbb{E}\eta(\tau)+\Delta \mathbf{V}^{T}(\tau)\mathbf{R}\Delta \mathbf{V}(\tau)+[\mathbb{E}\Delta \mathbf{V}(\tau)]^{T}\bar{R}\mathbb{E}\Delta \mathbf{V}(\tau)\big]\notag\\
&+\mathbb{E}[\eta^{T}(\Gamma+1)\Phi(\Gamma+1)\eta(\Gamma+1)]+[\mathbb{E}\eta(\Gamma+1)]^{T}\bar{\Phi}(\Gamma+1)\mathbb{E}\eta(\Gamma+1),
\end{align}
in which $\eta(\tau)=\frac{x^{\varepsilon}(\tau)-x(\tau)}{\varepsilon}$, and the costate $\theta(\tau)$ satisfies the following backward iteration
\begin{align}\label{equ:11}
\theta(\tau-1)=&Qx(\tau)+\bar{Q}\mathbb{E}x(\tau)+\mathbb{E}[\mathcal{A}^{T}(\tau)\theta(\tau)|\mathcal{F}_{0}(\tau)]+\mathbb{E}[\bar{\mathcal{A}}^{T}(\tau)\theta(\tau)],
\end{align}
with terminal condition $\theta(\Gamma)=\Phi(\Gamma+1)x(\Gamma+1)+\bar{\Phi}(\Gamma+1)\mathbb{E}x(\Gamma+1),$\ $\eta(\tau+1)$ satisfies the following iteration
\begin{align}\label{equ:10}
\eta(\tau+1)=&\mathcal{A}(\tau)\eta(\tau)+\bar{\mathcal{A}}(\tau)\mathbb{E}\eta(\tau)+\mathbf{B}(\tau)\Delta \mathbf{V}(\tau)+\bar{\mathbf{B}}(\tau)\mathbb{E}\Delta \mathbf{V}(\tau), \eta(0)=0.
\end{align}
\end{lemma}
\begin{proof}

\color{blue} {In view of space limitations, the detailed proof is omitted here, which can be deduced from \emph {Lemma 1 of \citep{qxz2020}. }} 
\end{proof}

Based on the results of Lemma \ref{lemma:1}, we can present the following lemma on the solvability conditions of Problem \ref{problem:1}.
\begin{lemma}\label{lemma:2}
Under Assumptions \ref{asss}-\ref{assumption:1}, Problem \ref{problem:1} can be uniquely solved if and only if the following equilibrium condition can be uniquely solvable for $0\leq i\leq h,i\leq \tau\leq \Gamma$,
\begin{align}\label{equ:15}
0=\mathcal{R}_{i}v_{i}(\tau)+\bar{\mathcal{R}}_{i}\mathbb{E}v_{i}(\tau)+\mathbb{E}[\mathcal{B}_{i}^{T}(\tau)\theta(\tau)|\mathcal{F}_{i}(\tau)]+\mathbb{E}[\bar{\mathcal{B}}_{i}^{T}(\tau)\theta(\tau)],
\end{align}
in which the costate $\theta(\tau)$ satisfies \eqref{equ:11}.\end{lemma}

\begin{proof}
`Necessity': If Problem \ref{problem:1} can be uniquely solved and Assumptions \ref{asss}-\ref{assumption:1} hold, we will show \eqref{equ:15} should be uniquely solved.

In fact, for any $\Delta v_{i}(\tau)$ and $\varepsilon\in\mathbb{R}$, if we denote $v_{i}(\tau)$ as the optimal control strategy  for $0\leq i\leq h,i\leq \tau\leq \Gamma$, then it can be implied from \eqref{equ:8} that,
\begin{align}\label{equ:16}
J_{\Gamma}(v^\varepsilon)-J_{\Gamma}(v)=\varepsilon^{2}\Delta J_{\Gamma}(\Delta v)+2\varepsilon\sum_{\tau=0}^{\Gamma}\sum_{i=0}^{h}\mathbb{E}\Big[\big[\mathcal{B}_{i}^{T}(\tau)\theta(\Gamma)+\mathbb{E}[\bar{\mathcal{B}}_{i}^{T}(\tau)\theta(\Gamma)]+\mathcal{R}_{i}v_{i}(\tau)+\bar{\mathcal{R}}_{i}\mathbb{E}v_{i}(\tau)\big]^{T}\Delta v_{i}(\tau)\Big]\geq0,
\end{align}

Note that $\Delta J_{\Gamma}(\Delta v)\geq 0$ can be shown from Assumption \ref{assumption:1}, then we will show \eqref{equ:15} holds by contraction. In other words, it is assumed that
\begin{align}\label{equ:17}
R_{i}&v_{i}(\tau)+\bar{R}_{i}\mathbb{E}v_{i}(\tau)+\mathbb{E}[\mathcal{B}_{i}^{T}(\tau)\theta(\tau)|\mathcal{F}_{i}(\tau)]+\mathbb{E}[\bar{\mathcal{B}}_{i}^{T}(\tau)\theta(\tau)]=\Theta_{i}(\tau)\neq 0.
\end{align}
By letting $\Delta v_{i}(\tau)=\Theta_{i}(\tau)$, we have
\begin{equation}
J_{\Gamma}(v^\varepsilon)-J_{\Gamma}(v)=\varepsilon^{2}\Delta J_{\Gamma}(\Delta v)+2\varepsilon \sum_{\tau=0}^{\Gamma}\sum_{i=0}^{h}\Theta_{i}^{T}(\tau)\Theta_{i}(\tau)\notag\\.
\end{equation}

Obviously, we can always find some $\varepsilon<0$ such that $ J_{\Gamma}(v^\varepsilon)-J_{\Gamma}(v)<0,$  which contradicts with (\ref{equ:15}). The proof of Necessity is straightforward.

`Sufficiency': Suppose (\ref{equ:15}) is uniquely solvable, we shall prove the uniquely solvable of Problem \ref{problem:1}. Actually, under Assumption \ref{assumption:1}, from (\ref{equ:8}) we know that for any $\varepsilon\in\mathbb{R}$ and $\Delta v_{i}(t),$  it follows that
\begin{equation}
J_{\Gamma}(v^\varepsilon)-J_{\Gamma}(v)=\varepsilon^{2}\Delta J_{\Gamma}(\Delta v)\geq0,\notag
\end{equation}
it is evident to see that Problem \ref{problem:1} is uniquely solvable.
\end{proof}

\begin{remark}\label{remark3}
Combining Lemma \ref{lemma:1} and Lemma \ref{lemma:2}, the necessary and sufficient
unique solvability conditions of Problem \ref{problem:1} are derived. Subsequently, in order to obtain the optimal decentralized control strategy, we must pay our attention on solving the following MF-FBSDEs,
\begin{align}\label{mffbsde}
\left\{
\begin{aligned}
x(\tau+1)=&\Big[\mathcal{A}x(\tau)+\bar{\mathcal{A}}\mathbb{E}x(\tau)+\sum_{i=0}^{h}[\mathcal{B}_{i}v_{i}(\tau)+\bar{\mathcal{B}}_{i}\mathbb{E}v_{i}(\tau)]\Big]\\
&+\Big[Cx(\tau)+\bar{C}\mathbb{E}x(\tau)+\sum_{i=0}^{h}[D_{i}v_{i}(\tau)+\bar{D}_{i}\mathbb{E}v_{i}(\tau)]\Big]\omega (\tau),\\
\theta(\tau-1)=&Qx(\tau)+\bar{Q}\mathbb{E}x(\tau)+\mathbb{E}[\mathcal{A}^{T}(\tau)\theta(\tau)|\mathcal{F}_{0}(\tau)]+\mathbb{E}[\bar{\mathcal{A}}^{T}(\tau)\theta(\tau)],\\
0=\mathcal{R}_{i}&v_{i}(\tau)+\bar{\mathcal{R}}_{i}\mathbb{E}v_{i}(\tau)+\mathbb{E}[\mathcal{B}_{i}^{T}(\tau)\theta(\tau)|\mathcal{F}_{i}(\tau)]+\mathbb{E}[\bar{\mathcal{B}}_{i}(\tau)^{T}\theta(\tau)],0\leq i\leq h,i\leq \tau\leq \Gamma,\\
\theta(\Gamma)=&\Phi(\Gamma+1)x(\Gamma+1)+\bar{\Phi}(\Gamma+1)\mathbb{E}x(\Gamma+1),
\end{aligned}
\right.
\end{align}
which consists of system (\ref{equ:2}), costate equation (\ref{equ:11}) and equilibrium condition (\ref{equ:15}), where (\ref{equ:2}) is forward and (\ref{equ:11}) is backward.
\end{remark}

\section{Optimal Decentralized Control}
In this section, we will explore the optimal decentralized control strategy by solving the MF-FBSDEs \eqref{mffbsde}. For this reason, an innovative orthogonal decomposition approach will be adopted.

\subsection{Orthogonal Decomposition Approach}

By applying the orthogonal decomposition approach, the following lemma can be presented.
\begin{lemma}\label{lemma:4}
System (\ref{equ:2}) and cost functional (\ref{equ:4}) can be rewritten as, respectively
\begin{align}\label{equ:19}
x(\tau&+1)=\mathcal{A}(\tau)x(\tau)+\bar{\mathcal{A}}(\tau)\mathbb{E}x(\tau)+\mathbf{B}_{i}(\tau)\mathbf{V}_{i}(\tau)+\mathbf{\bar{B}}_{i}(\tau)\mathbb{E}\mathbf{V}_{i}(\tau)\notag\\
&+\sum_{j=0}^{i-1}[\mathcal{B}_{j}(\tau)\tilde{v}_{j,i}(\tau)+\bar{\mathcal{B}}_{j}(\tau)\mathbb{E}\tilde{v}_{j,i}(\tau)]+\sum_{j=i+1}^{h}[\mathcal{B}_{j}(\tau)v_{j}(\tau)+\bar{\mathcal{B}}_{j}(\tau)\mathbb{E}v_{j}(\tau)],
\end{align}
and
\begin{align}\label{equ:20} J_{\Gamma}(v)&=\sum_{\tau=0}^{\Gamma}\mathbb{E}\Big[x^{T}(\tau)Qx(\tau)+[\mathbb{E}x(\tau)]^{T}\bar{Q}\mathbb{E}x(\tau)+\mathbf{V}_{i}^{T}(\tau)\mathbf{R}_{i}\mathbf{V}_{i}(\tau)+[\mathbb{E}\mathbf{V}_{i}(\tau)]^{T}\mathbf{\bar{R}}_{i}\mathbb{E}\mathbf{V}_{i}(\tau)\notag\\
&+\sum_{j=0}^{i-1}\big[\tilde{v}_{j,i}^{T}(\tau)\mathcal{R}_{j}\tilde{v}_{j,i}(\tau)+[\mathbb{E}\tilde{v}_{j,i}(\tau)]^{T}\bar{\mathcal{R}}_{j}\mathbb{E}\tilde{v}_{j,i}(\tau)\big]+\sum_{j=i+1}^{h}\big[v_{j}^{T}(\tau)\mathcal{R}_{j}v_{j}(\tau)+[\mathbb{E}v_{j}(\tau)]^{T}\bar{\mathcal{R}}_{j}\mathbb{E}v_{j}(\tau)\big]\Big]\notag\\
&+\mathbb{E}[x^T(\Gamma+1)\Phi(\Gamma+1)x(\Gamma+1)]+[\mathbb{E}x(\Gamma+1)]^T\bar{\Phi}(\Gamma+1)\mathbb{E}x(\Gamma+1),
\end{align}
where for $0\leq j\leq i-1,$ there holds
\begin{align}\label{equ:21}
\hat{v}_{j,i}(\tau)=&\mathbb{E}[v_{j}(\tau)\mid \mathcal{F}_{i}(\tau)],\notag\\
\tilde{v}_{j,i}(\tau)=&v_{j}(\tau)-\hat{v}_{j,i}(\tau).
\end{align}
Moreover, (\ref{equ:11}) and (\ref{equ:15}) can be rewritten as follows, respectively
\begin{align}\label{equ:22}
0=&\mathbf{R}_{i}\mathbf{V}_{i}(\tau)+\mathbf{\bar{R}}_{i}\mathbb{E}\mathbf{V}_{i}(\tau)+\mathbb{E}[\mathbf{B}^{T}(\tau)\theta(\tau)\mid \mathcal{F}_{i}(\tau)]+\mathbb{E}[\mathbf{\bar{B}}^{T}(\tau)\theta(\tau)],
\end{align}
\begin{align}\label{equ:23}
0=&\mathcal{R}_{j}\tilde{v}_{j,i}(\tau)+\bar{\mathcal{R}}_{j}\mathbb{E}\tilde{v}_{j,i}(\tau)-\mathbb{E}[\mathcal{B}_{j}^{T}(\tau)\theta(\tau)\mid \mathcal{F}_{i}(\tau)]+\mathbb{E}[\mathcal{B}_{j}^{T}(\tau)\theta(\tau)|\mathcal{F}_{j}(\tau)],0\leq j\leq i-1,
\end{align}
\begin{align}\label{equ:24}
0=&\mathcal{R}_{j}v_{j}(\tau)+\bar{\mathcal{R}}_{j}\mathbb{E}v_{j}(\tau)+\mathbb{E}[\mathcal{B}_{j}^{T}(\tau)\theta(\tau)|\mathcal{F}_{j}(\tau)]+\mathbb{E}[\bar{\mathcal{B}}_{j}^{T}(\tau)\theta(\tau)],i\leq j\leq h,
\end{align}
where $0\leq i\leq h,$ and $i\leq \tau\leq \Gamma.$
\end{lemma}
\begin{proof}
By plugging (\ref{equ:15}) into (\ref{equ:1}) and (\ref{equ:21}), we get:
\begin{align}
0=&\mathcal{R}_{j}\mathbb{E}[v_{j}(\tau)\mid \mathcal{F}_{i}(\tau)]+\bar{\mathcal{R}}_{j}\mathbb{E}[\mathbb{E}v_{j}(\tau)\mid \mathcal{F}_{i}(\tau)]\notag\\
&+\mathbb{E}\Big[\mathbb{E}[\mathcal{B}_{j}^{T}(\tau)\theta(\tau)\mid \mathcal{F}_{j}(\tau)]\mid \mathcal{F}_{i}(\tau)\Big]+\mathbb{E}\Big[\mathbb{E}[\bar{\mathcal{B}}_{j}^{T}(\tau)\theta(\tau)]\mid \mathcal{F}_{i}(\tau)\Big].\notag
\end{align}
Then, for $0\leq j\leq i-1,$
\begin{align}
0=&\mathcal{R}_{j}\hat{v}_{j,i}(\tau)+\bar{\mathcal{R}}_{j}\mathbb{E}\hat{v}_{j,i}(\tau)+\mathbb{E}[\mathcal{B}_{j}^{T}(\tau)\theta(\tau)\mid \mathcal{F}_{i}(\tau)] +\mathbb{E}[\bar{\mathcal{B}}_{j}^{T}(\tau)\theta(\tau)],\notag
\end{align}
for simplicity of calculation, we have
\begin{align}\label{equ:25}
\mathcal{R}_{j}\hat{v}_{j,i}(\tau)&+\bar{\mathcal{R}}_{j}\mathbb{E}\hat{v}_{j,i}(\tau)=-\mathbb{E}[\mathcal{B}_{j}^{T}(\tau)\theta(\tau)\mid \mathcal{F}_{i}(\tau)] -\mathbb{E}[\bar{\mathcal{B}}_{j}^{T}(\tau)\theta(\tau)].
\end{align}
What's more, for $i\leq j\leq h,$ obviously (\ref{equ:23}) is established.

We know that (\ref{equ:23}) can be obtained by combining (\ref{equ:15}) and (\ref{equ:25}). Next, using (\ref{equ:1}) and (\ref{equ:25}), (\ref{equ:22}) can be verified.
Before proving (\ref{equ:19}) and (\ref{equ:20}), it is worth noting that the following two relationships are established:

1) For $0\leq j\leq i-1,$ we infer that
\begin{align}
\mathbb{E}[v_{i}^{T}(\tau)T_{i,j}\tilde{v}_{j,i}(\tau)]=&\mathbb{E}\Big[\mathbb{E}[v_{i}^{T}(\tau)T_{i,j}\tilde{v}_{j,i}(\tau)\mid\mathcal{F}_{i}(\tau)]\Big]=\mathbb{E}\Big[v_{i}^{T}(\tau)T_{i,j}\mathbb{E}[v_{j}(\tau)-\hat{v}_{j,i}(\tau)\mid\mathcal{F}_{i}(\tau)]\Big]\notag\\
=&\mathbb{E}\bigg[v_{i}^{T}(\tau)T_{i,j}\mathbb{E}\Big[v_{j}(\tau)-\mathbb{E}[v_{j}(\tau)\mid\mathcal{F}_{i}(\tau)]\mid\mathcal{F}_{i}(\tau)\Big]\bigg]=0.\notag
\end{align}

2) For $0\leq j,n\leq i-1,$ there holds that
\begin{align}
\mathbb{E}&[\hat{v}_{n,i}^{T}(\tau)T_{n,j}\tilde{v}_{j,i}(\tau)]=\mathbb{E}\Big[\mathbb{E}[\hat{v}_{n,i}^{T}(\tau)T_{n,j}\tilde{v}_{j,i}(\tau)\mid\mathcal{F}_{i}(\tau)]\Big]\notag\\
=&\mathbb{E}\Big[\hat{v}_{n,i}^{T}(\tau)T_{n,j}\mathbb{E}[\tilde{v}_{j,i}(\tau)\mid\mathcal{F}_{i}(\tau)]\Big]=0.\notag
\end{align}

In summary, we have completed the proof of the orthogonality of $\mathbf{V}_{i}(\tau)$ and $\tilde{v}_{0,i}(\tau),\cdots,\tilde{v}_{i-1,i}(\tau)$
for $0\leq i\leq h,$ and then we can immediately verify that (\ref{equ:19}) and (\ref{equ:20}) hold.
\end{proof}

\begin{remark}\label{remark4}
The decoupling of the MF-FBSDEs \eqref{mffbsde} is difficult in view of the non-classic information structure of Problem \ref{problem:1}. Therefore, the orthogonal decomposition approach is introduced in Lemma \ref{lemma:4}, which will play a critical role in decoupling \eqref{mffbsde}. In what follows, we will show the method to derive the optimal decentralized control strategy by the use of the orthogonal decomposition approach to decouple MF-FBSDEs \eqref{mffbsde}.
\end{remark}

In view of the delayed information pattern for the $h+1$ controllers, it is necessary to derive the associated mathematical conditional expectation, i.e., the optimal predictor.

\begin{lemma}\label{lemma:5}
The optimal predictor $\hat{x}_{\tau/\tau-i}=\mathbb{E}[x(\tau)\mid \mathcal{F}_{i}(\tau)]=\mathbb{E}[x(\tau)\mid \mathcal{F}_{0}(\tau-i)]$ can be given as follows:
\begin{align}\label{equ:26}
\left\{
\begin{aligned}
\hat{x}_{\tau/\tau-i}&=\mathcal{A}^{\tau}x(0)+[(\mathcal{A}+\bar{\mathcal{A}})^{\tau}-\mathcal{A}^{\tau}]\mathbb{E}x(0)+\sum_{j=1}^{\tau}\begin{bmatrix}\mathcal{A}^{j-1}&(\mathcal{A}+\bar{\mathcal{A}})^{j-1}-\mathcal{A}^{j-1}\end{bmatrix}\big\{\begin{bmatrix}
\mathbf{B}_{i-j} & \mathbf{\bar{B}}_{i-j}\\
0 & \mathbf{B}_{i-j}+\mathbf{\bar{B}}_{i-j}
\end{bmatrix}\\
&\times\begin{bmatrix}
\mathbf{V}_{i-j}(\tau-j)\\
\mathbb{E}\mathbf{V}_{i-j}(\tau-j)
\end{bmatrix}+\sum_{m=i-j+1}^{h}\begin{bmatrix}
\mathcal{B}_{m} & \bar{\mathcal{B}}_{m}\\
0 & \mathcal{B}_{m}+\bar{\mathcal{B}}_{m}
\end{bmatrix}\begin{bmatrix}
v_{m}(\tau-j)\\
\mathbb{E}v_{m}(\tau-j)
\end{bmatrix}\big\},i=1,\cdots,h,\\
\hat{x}_{\tau+1/\tau-i}&=\mathcal{A}\hat{x}_{\tau/\tau-i}+\bar{\mathcal{A}}\mathbb{E}x(\tau)+\mathbf{B}_{i}\mathbf{V}_{i}(\tau)+\mathbf{\bar{B}}_{i}\mathbb{E}\mathbf{V}_{i}(\tau)+\sum_{m=i+1}^{h}[\mathcal{B}_{m}v_{m}(\tau)+\bar{\mathcal{B}}_{m}\mathbb{E}v_{m}(\tau)],i=0,\cdots,h-1,\\
\hat{x}_{\tau+1/\tau-h}&=\mathcal{A}\hat{x}_{\tau/\tau-h}+\bar{\mathcal{A}}\mathbb{E}x(\tau)+\mathbf{B}_{h}\mathbf{V}_{h}(\tau)+\mathbf{\bar{B}}_{h}\mathbb{E}\mathbf{V}_{h}(\tau).
\end{aligned}
\right.
\end{align}

\end{lemma}

\begin{proof}
For (\ref{equ:2}), it can be easily induced that
by taking the mathematical expectation of system (\ref{equ:2}):
\begin{align}\label{equ:27}
\mathbb{E}x(\tau+1)=&(\mathcal{A}+\bar{\mathcal{A}})\mathbb{E}x(\tau)+\sum_{i=0}^{h}(\mathcal{B}_{i}+\bar{\mathcal{B}}_{i})\mathbb{E}v_{i}(\tau).
\end{align}
Subsequently,

\begin{align}\label{equ:28}
x(\tau)=&\begin{bmatrix}
\mathcal{A}(\tau-1) & \bar{\mathcal{A}}(\tau-1)
\end{bmatrix}\begin{bmatrix}
x(\tau-1)\\
\mathbb{E}x(\tau-1)
\end{bmatrix}+\sum_{m=0}^{h}\begin{bmatrix}
\mathcal{B}_{m}(\tau-1) & \bar{\mathcal{B}}_{m}(\tau-1)
\end{bmatrix}\begin{bmatrix}
v_{m}(\tau-1)\\
\mathbb{E}v_{m}(\tau-1)
\end{bmatrix}\notag\\
=&\begin{bmatrix}
\mathcal{A}(\tau-1) & \bar{\mathcal{A}}(\tau-1)
\end{bmatrix}\begin{bmatrix}
\mathcal{A}(\tau-2) & \bar{\mathcal{A}}(\tau-2)\\
0 & \mathcal{A}+\bar{\mathcal{A}}
\end{bmatrix}\times \cdots \times\begin{bmatrix}
\mathcal{A}(0) & \bar{\mathcal{A}}(0)\\
0 & \mathcal{A}+\bar{\mathcal{A}}
\end{bmatrix}\begin{bmatrix}
x(0)\\
\mathbb{E}x(0)
\end{bmatrix}\notag\\
&+\sum_{j=1}^{\tau}\sum_{m=0}^{h}\begin{bmatrix}
\mathcal{A}(\tau-1) & \bar{\mathcal{A}}(\tau-1)
\end{bmatrix}\begin{bmatrix}
\mathcal{A}(\tau-2) & \bar{\mathcal{A}}(\tau-2)\\
0 & \mathcal{A}+\bar{\mathcal{A}}
\end{bmatrix}\times\cdots \times\begin{bmatrix}
\mathcal{A}(\tau-j+1) & \bar{\mathcal{A}}(\tau-j+1)\\
0 & \mathcal{A}+\bar{\mathcal{A}}
\end{bmatrix}\notag\\
&\times\begin{bmatrix}
\mathcal{B}_{m}(\tau-j) & \bar{\mathcal{B}}_{m}(\tau-j)\\
0 & \mathcal{B}_{m}+\bar{\mathcal{B}}_{m}
\end{bmatrix}\begin{bmatrix}
v_{m}(\tau-j)\\
\mathbb{E}v_{m}(\tau-j)
\end{bmatrix}\notag\\
=&\begin{bmatrix}
\mathcal{A}(\tau-1) & \bar{\mathcal{A}}(\tau-1)
\end{bmatrix}\begin{bmatrix}
\mathcal{A}(\tau-2) & \bar{\mathcal{A}}(\tau-2)\\
0 & \mathcal{A}+\bar{\mathcal{A}}
\end{bmatrix}\times\cdots \times\begin{bmatrix}
\mathcal{A}(0) & \bar{\mathcal{A}}(0)\\
0 & \mathcal{A}+\bar{\mathcal{A}}
\end{bmatrix}\begin{bmatrix}
x(0)\\
\mathbb{E}x(0)
\end{bmatrix}\notag\\
&+\sum_{j=1}^{\tau}\sum_{m=0}^{i-j}\begin{bmatrix}
\mathcal{A}(\tau-1) & \bar{\mathcal{A}}(\tau-1)
\end{bmatrix}\begin{bmatrix}
\mathcal{A}(\tau-2) & \bar{\mathcal{A}}(\tau-2)\\
0 & \mathcal{A}+\bar{\mathcal{A}}
\end{bmatrix}\times\cdots \times\begin{bmatrix}
\mathcal{A}(\tau-j+1) & \bar{\mathcal{A}}(\tau-j+1)\\
0 & \mathcal{A}+\bar{\mathcal{A}}
\end{bmatrix}\notag\\
&\times\begin{bmatrix}
\mathcal{B}_{m}(\tau-j) & \bar{\mathcal{B}}_{m}(\tau-j)\\
0 & \mathcal{B}_{m}+\bar{\mathcal{B}}_{m}
\end{bmatrix}\begin{bmatrix}
v_{m}(\tau-j)\\
\mathbb{E}v_{m}(\tau-j)
\end{bmatrix}+\sum_{j=1}^{\tau}\sum_{m=i-j+1}^{h}\begin{bmatrix}
\mathcal{A}(\tau-1) & \bar{\mathcal{A}}(\tau-1)
\end{bmatrix}\notag\\
&\times\begin{bmatrix}
\mathcal{A}(\tau-2) & \bar{\mathcal{A}}(\tau-2)\\
0 & \mathcal{A}+\bar{\mathcal{A}}
\end{bmatrix}\times\cdots \times\begin{bmatrix}
\mathcal{A}(\tau-j+1) & \bar{\mathcal{A}}(\tau-j+1)\\
0 & \mathcal{A}+\bar{\mathcal{A}}
\end{bmatrix}\begin{bmatrix}
\mathcal{B}_{m}(\tau-j) & \bar{\mathcal{B}}_{m}(\tau-j)\\
0 & \mathcal{B}_{m}+\bar{\mathcal{B}}_{m}
\end{bmatrix}\begin{bmatrix}
v_{m}(\tau-j)\\
\mathbb{E}v_{m}(\tau-j)
\end{bmatrix}.
\end{align}

By taking the conditional mathematical expectation of $x(\tau)$, the optimal predictor $\mathbb{E}[x(\tau)\mid \mathcal{F}_{i}(\tau)]$ can be derived as
\begin{align}
\hat{x}_{\tau/\tau-i}=&\mathbb{E}[x(\tau)\mid \mathcal{F}_{i}(\tau)]\notag\\
=&\mathcal{A}^{\tau}x(0)+[(\mathcal{A}+\bar{\mathcal{A}})^{\tau}-\mathcal{A}^{\tau}]\mathbb{E}x(0)+\sum_{j=1}^{\tau}\begin{bmatrix}\mathcal{A}^{j-1}&(\mathcal{A}+\bar{\mathcal{A}})^{j-1}-\mathcal{A}^{j-1}\end{bmatrix}\big\{\begin{bmatrix}
\mathbf{B}_{i-j} & \mathbf{\bar{B}}_{i-j}\\
0 & \mathbf{B}_{i-j}+\mathbf{\bar{B}}_{i-j}
\end{bmatrix}\notag\\
&\times\begin{bmatrix}
\mathbf{V}_{i-j}(\tau-j)\\
\mathbb{E}\mathbf{V}_{i-j}(\tau-j)
\end{bmatrix}+\sum_{m=i-j+1}^{h}\begin{bmatrix}
\mathcal{B}_{m} & \bar{\mathcal{B}}_{m}\\
0 & \mathcal{B}_{m}+\bar{\mathcal{B}}_{m}
\end{bmatrix}\begin{bmatrix}
v_{m}(\tau-j)\\
\mathbb{E}v_{m}(\tau-j)
\end{bmatrix}\big\},i=1,\cdots,h.\notag
\end{align}

Note that the following relationship can be holds,
\begin{align}
\mathbb{E}[x(\tau)\mid \mathcal{F}_{i}(\tau)]=&\mathbb{E}[x(\tau)\mid \mathcal{F}_{0}(\tau-i)]=\mathbb{E}[x(\tau)\mid \mathcal{F}_{i-j}(\tau-j)].\notag
\end{align}
Moreover, it is noted that $v_{m}(\tau-j)$ is $\mathcal{F}_{i-j}(\tau-j)$-measurable for $i-j+1\leq m\leq h$. Otherwise, $v_{m}(\tau-j)$ is not $\mathcal{F}_{i-j}(\tau-j)$-measurable for $m=0,\cdots,i-j$, subsequently, by using (\ref{equ:21}), we know that $\mathbb{E}[v_{m}(\tau-j)\mid\mathcal{F}_{i-j}(\tau-j)]$ can be represented by $\hat{v}_{m,i-j}(\tau-j)$ for $m=0,\cdots,i-j.$

Combining the properties of Gaussian white noise and  $v_{i}(\tau)$ is $\mathcal{F}_{i}(\tau)$-measurable, and by using (\ref{equ:27}) and (\ref{equ:28}) we can immediately verify that the relationship in (\ref{equ:26}) holds.
\end{proof}
\subsection{Decoupling MF-FBSDEs}
In this section, we will adopt the induction approach to decouple the MF-FBSDEs \eqref{mffbsde} and derive a solution to Problem \ref{problem:1}.

For simplicity, let's define the following asymmetric Riccati equations,
\begin{align}\label{equ:101}
\left\{
\begin{aligned}
\Phi(\tau)=&Q+\mathcal{A}^{T}\Phi(\tau+1)\mathcal{A}+\sigma ^{2}C^{T}\Phi(\tau+1)C\\
&-[\bar{L}_{0}^{T}(\tau)+\mathcal{A}^{T}\varphi_{1}(\tau+1)\mathcal{B}_{0}]\bar{\mathbf{\Upsilon}} _{0}^{-1}(\tau)\bar{Y}_{0,0}(\tau)+\mathcal{A}^{T}\varphi_{1}(\tau+1)\mathcal{A},\\
\bar{\Phi}(\tau)=&\bar{Q}+(\mathcal{A}+\bar{\mathcal{A}})^{T}[\Phi(\tau+1)+\bar{\Phi}(\tau+1)+\sum_{j=1}^{h}\varphi_{j}(\tau+1)](\mathcal{A}+\bar{\mathcal{A}})\\
&-\mathcal{A}^{T}[\Phi(\tau+1)+\sum_{j=1}^{h}\varphi_{j}(\tau+1)]\mathcal{A}+\sigma ^{2}(C+\bar{C})^{T}\Phi(\tau+1)(C+\bar{C})\\
&-\sigma ^{2}C^{T}\Phi(\tau+1)C-\sum_{i=0}^{h}\Big[[L_{i}^{T}(\tau)\mathbf{I}_{i}+(\mathcal{A}+\bar{\mathcal{A}})^{T}\varphi_{i+1}(\tau+1)\\
&\times(\mathbf{B}_{i}+\mathbf{\bar{B}}_{i})]\sum_{j=i}^{h}\mathbf{\Upsilon} _{i}^{-1}(\tau)Y_{i,j}(\tau)-[\bar{L}_{i}^{T}(\tau)\mathbf{I}_{i}+\mathcal{A}^{T}\varphi_{i+1}(\tau+1)\mathbf{B}_{i}]\sum_{j=i}^{h}\mathbf{\bar{\Upsilon}} _{i}^{-1}(\tau)\bar{Y}_{i,j}(\tau)\Big],\\
\varphi_{j}(\tau)=&-\sum_{i=0}^{j}[\bar{L}_{i}^{T}(\tau)\mathbf{I}_{i}+\mathcal{A}^{T}\varphi_{i+1}(\tau+1)\mathbf{B}_{i}]\mathbf{\bar{\Upsilon}} _{i}^{-1}(\tau)\sum_{j=i}^{h}\bar{Y}_{i,j}(\tau)+\mathcal{A}^{T}\varphi_{j+1}(\tau+1)\mathcal{A}, 1\leq j\leq h,
\end{aligned}
\right.
\end{align}
with terminal conditions $\Phi(\Gamma+1), \bar{\Phi}(\Gamma+1)$ and $\varphi_{j}(\Gamma+1)=0,\varphi_{h+1}(\tau)=0$ for $1\leq j\leq h$ and $0\leq \tau\leq \Gamma.$ In the above, for $1\leq i\leq h,i+1\leq j\leq h,$

\begin{align}\label{equ:29}
\mathbf{\bar{\Upsilon}} _{i}(\tau)=&\mathbf{R}_{i}+\mathbf{B}_{i}^{T}[\Phi(\tau+1)+\sum_{j=1}^{h}\varphi_{j}(\tau+1)]\mathbf{B}_{i}+\sigma ^{2}\mathbf{D}_{i}\Phi(\tau+1)\mathbf{D}_{i},\notag\\
\mathbf{\Upsilon}_{i}(\tau)=&\mathbf{R}_{i}+\mathbf{\bar{R}}_{i}+(\mathbf{B}_{i}+\mathbf{\bar{B}}_{i})^{T}[\Phi(\tau+1)+\bar{\Phi}(\tau+1)+\sum_{j=1}^{h}\varphi_{j}(\tau+1)]\notag\\
&\times(\mathbf{B}_{i}+\mathbf{\bar{B}}_{i})+\sigma ^{2}(\mathbf{D}_{i}+\mathbf{\bar{D}}_{i})\Phi(\tau+1)(\mathbf{D}_{i}+\mathbf{\bar{D}}_{i}),\notag\\
\bar{Y}_{i,i}(\tau)=&\mathbf{B}_{i}^{T}[\Phi(\tau+1)+\sum_{j=1}^{h}\varphi_{j}(\tau+1)]\mathcal{A}+\sigma ^{2}\mathbf{D}_{i}^{T}\Phi(\tau+1)C,\notag\\
Y_{i,i}(\tau)=&(\mathbf{B}_{i}+\mathbf{\bar{B}}_{i})^{T}[\Phi(\tau+1)+\bar{\Phi}(\tau+1)+\sum_{j=1}^{h}\varphi_{j}(\tau+1)](\mathcal{A}+\bar{\mathcal{A}})\notag\\
&+\sigma ^{2}(\mathbf{D}_{i}+\mathbf{\bar{D}}_{i})^{T}\Phi(\tau+1)(C+\bar{C}),\notag\\
\bar{Y}_{i,j}(\tau)=&-\big\{\mathbf{B}_{i}^{T}[\Phi(\tau+1)+\sum_{m=1}^{h}\varphi_{m}(\tau+1)]\mathcal{B}_{j}+\sigma ^{2}\mathbf{D}_{i}^{T}\Phi(\tau+1)D_{j}\big\}\notag\\
&\times\mathbf{I}_{j}\mathbf{\bar{\Upsilon}}_{j}^{-1}(\tau)\sum_{m=j}^{h}\bar{Y}_{j,m}(\tau),\notag\\
Y_{i,j}(\tau)=&-\big\{(\mathbf{B}_{i}+\mathbf{\bar{B}}_{i})^{T}[\Phi(\tau+1)+\bar{\Phi}(\tau+1)+\sum_{m=1}^{h}\varphi_{m}(\tau+1)](\mathcal{B}_{j}+\bar{\mathcal{B}}_{j})\notag\\
&+\sigma ^{2}(\mathbf{D}_{i}+\mathbf{\bar{D}}_{i})^{T}\Phi(\tau+1)(D_{j}+\bar{D}_{j})\big\}\mathbf{I}_{j}\mathbf{\Upsilon}_{j}^{-1}(\tau)\sum_{m=j}^{h}Y_{j,m}(\tau), i+1\leq j\leq h,
\end{align}
where
\begin{align}
\mathbf{I}_{i}=&[0,\cdots,0,I_{m_{i}}], \mathbf{I}_{0}=I_{m_{0}},\notag
\end{align}
\begin{align}
L_{i}(\tau)=&(\mathcal{B}_{i}+\bar{\mathcal{B}}_{i})^{T}[\Phi(\tau+1)+\bar{\Phi}(\tau+1)+\sum_{j=1}^{h}\varphi_{j}(\tau+1)](\mathcal{A}+\bar{\mathcal{A}})\notag\\
&+\sigma ^{2}(D_{i}+\bar{D}_{i})^{T}\Phi(\tau+1)(C+\bar{C}),1\leq i\leq h,\notag
\end{align}
\begin{align}
\bar{L}_{i}(\tau)=&\mathcal{B}_{i}^{T}[\Phi(\tau+1)+\sum_{j=1}^{h}\varphi_{j}(\tau+1)]\mathcal{A}+\sigma ^{2}D_{i}^{T}\Phi(\tau+1)C,\notag
\end{align}
\begin{align}
L_{0}(\tau)=&(\mathcal{B}_{0}+\bar{\mathcal{B}}_{0})^{T}[\Phi(\tau+1)+\bar{\Phi}(\tau+1)](\mathcal{A}+\bar{\mathcal{A}})+\sigma ^{2}(D_{0}+\bar{D}_{0})^{T}\Phi(\tau+1)(C+\bar{C}),\notag
\end{align}
\begin{align}
\bar{L}_{0}(\tau)=&\mathcal{B}_{0}^{T}\Phi(\tau+1)\mathcal{A}+\sigma ^{2}D_{0}^{T}\Phi(\tau+1)C.\notag
\end{align}

Now it is the position to state the solution to Problem \ref{problem:1}.
\begin{theorem}\label{theorem:2}
Suppose Assumptions \ref{asss}-\ref{assumption:1} hold, for (\ref{equ:2}) and (\ref{equ:4}), Problem \ref{problem:1} is uniquely solvable if and only if $\mathbf{\Upsilon}_{i}(\tau)$  and $\mathbf{\bar{\Upsilon}}_{i}(\tau)$ are invertible, for $0\leq i\leq h,i\leq \tau\leq \Gamma,$ where $\mathbf{\Upsilon}_{i}(\tau)$ and $\mathbf{\bar{\Upsilon}}_{i}(\tau)$ are given in (\ref{equ:29}).

In this case, the unique optimal decentralized control strategy $v_{i}(\tau)$ can be calculated as
\begin{align}\label{equ:30}
v_{i}(\tau)=&\mathbf{I}_{i}\mathbf{V}_{i}(\tau)\notag\\
=&-\mathbf{I}_{i}\mathbf{\bar{\Upsilon}} _{i}^{-1}(\tau)\sum_{j=i}^{h}\bar{Y}_{i,j}(\tau)\hat{x}_{\tau/\tau-j}-\sum_{j=i}^{h}\mathbf{I}_{i}[\mathbf{\Upsilon} _{i}^{-1}(\tau)Y_{i,j}(\tau)-\mathbf{\bar{\Upsilon}} _{i}^{-1}(\tau)\bar{Y}_{i,j(\tau)}]\mathbb{E}x(\tau),i\leq \tau\leq \Gamma,
\end{align}
where $\hat{x}_{\tau/\tau-j}$ can be calculated from Lemma \ref{lemma:5}, and $Y_{i,j}(\tau),\bar{Y}_{i,j}(\tau)$ are given in (\ref{equ:29}). Furthermore, the corresponding optimal cost functional is given by
\begin{align}\label{equ:33} J_{\Gamma}^{*}(v)=&x^{T}(0)[\Phi(0)+\bar{\Phi}(0)+\sum_{j=1}^{h}\varphi_{j}(0)]x(0)+\sum_{i=0}^{h}\sum_{\tau=0}^{i-1}v_{i}^T(\tau)[\mathcal{R}_{i}+\bar{\mathcal{R}}_{i}]v_{i}(\tau).
\end{align}

The relationship between the state $x(\tau)$ and the costate $\theta(\tau)$ (i.e., the solution to MF-FBSDEs (\ref{equ:19}) and (\ref{equ:22})-(\ref{equ:24})) can be {given} as
\begin{align}\label{equ:32}
\theta(\tau)=&\Phi(\tau+1)x(\tau+1)+\bar{\Phi}(\tau+1)\mathbb{E}x(\tau+1)+\sum_{j=1}^{h}\varphi_{j}(\tau+1)\hat{x}_{\tau+1/\tau+1-j}, -1\leq \tau\leq \Gamma.
\end{align}
\end{theorem}
\begin{proof}
Please refer to the Appendix for detailed proof.
\end{proof}

\begin{remark}
In Theorem \ref{theorem:2}, it is shown that the gain matrices of the optimal decentralized control strategy can be calculated via some asymmetric Riccati equations, which can be calculated offline. Moreover, we prove that the optimal control strategy is a linear feedback of the optimal predictor $\hat{x}_{\tau/\tau-j}$ and the state mean $\mathbb{E}x(\tau)$, which is feasible in calculation. The obtained results are derived for the first time.
\end{remark}

\begin{remark}
From the main results of Theorem \ref{theorem:2}, it is not hard to verify that the obtained results include the results of \citep{zqf2019, qxz2020} as special cases. On one hand, with $h=0$, Problem \ref{problem:1} can be reduced as the control problem for mean-field system without delay in \citep{zqf2019}, and the optimal control with $h=0$ in Theorem \ref{theorem:2} is exactly the optimal control shown in \citep{zqf2019}. On the other hand, with $\bar{\mathcal{A}}=0,\bar{\mathcal{B}}_{i}=0,\bar{C}=0,\bar{D}_{i}=0$, Problem \ref{problem:1} turns into the case investigated in \citep{qxz2020}, and the optimal decentralized control strategy in Theorem \ref{theorem:2} can be presented the same as that in \citep{qxz2020}.
\end{remark}

\section{Numerical Example}
To illustrate the obtained results in Theorem \ref{theorem:2}, the following numerical example shall be given as below.

Without loss of generality, we consider (\ref{equ:2}) and (\ref{equ:4}) with
$n=2,h=2,m_{0}=m_{1}=m_{2}=2,\Gamma=5,$
\begin{align}
\mathcal{A}=&\begin{bmatrix}
0.6&0.3\\
0.4&0.2
\end{bmatrix},\bar{\mathcal{A}}=\begin{bmatrix}
0.2&0\\
0&-0.6
\end{bmatrix},
C=\begin{bmatrix}
0.6&0.3\\
0.4&0.2
\end{bmatrix},\bar{C}=\begin{bmatrix}
0.2&0\\
0&-0.6
\end{bmatrix},\notag\\
\mathcal{B}_{0}=&\begin{bmatrix}
0.3&0.2\\
0.4&-0.1
\end{bmatrix},\mathcal{B}_{1}=\begin{bmatrix}
0.1&0.2\\
0&0.1
\end{bmatrix},\mathcal{B}_{2}=\begin{bmatrix}
0.3&0.1\\
0.5&0.8
\end{bmatrix},\notag\\
D_{0}=&\begin{bmatrix}
0.3&0.2\\
0.4&-0.1
\end{bmatrix},D_{1}=\begin{bmatrix}
0.1&0.2\\
0&0.1
\end{bmatrix},D_{2}=\begin{bmatrix}
0.3&0.1\\
0.5&0.8
\end{bmatrix},\notag\\
\bar{\mathcal{B}}_{0}=&\begin{bmatrix}
-0.4&0.8\\
0.2&0.9
\end{bmatrix},\bar{\mathcal{B}}_{1}=\begin{bmatrix}
0&0.1\\
0&-1
\end{bmatrix},\bar{\mathcal{B}}_{2}=\begin{bmatrix}
-0.8&0\\
0&0.2
\end{bmatrix},\notag\\
\bar{D}_{0}=&\begin{bmatrix}
-0.4&0.8\\
0.2&0.9
\end{bmatrix},\bar{D}_{1}=\begin{bmatrix}
0&0.1\\
0&-1
\end{bmatrix},\bar{D}_{2}=\begin{bmatrix}
-0.8&0\\
0&0.2
\end{bmatrix},\notag\\
Q=&I_{2},\mathcal{R}_{0}=I_{2},\mathcal{R}_{1}=I_{2},\mathcal{R}_{2}=I_{2},\bar{Q}=I_{2},\bar{\mathcal{R}}_{0}=I_{2},\bar{\mathcal{R}}_{1}=I_{2},\bar{\mathcal{R}}_{2}=I_{2},\notag\\
\Phi(\Gamma&+1)=I_{2},\bar{\Phi}(\Gamma+1)=I_{2},
\varphi_{1}(\Gamma+1)=0,\varphi_{2}(\Gamma+1)=0.\notag
\end{align}

By using Theorem \ref{theorem:2}, $\mathbf{\Upsilon} _{i}(\tau), \mathbf{\bar{\Upsilon}} _{i}(\tau), Y_{i,j}(\tau), \bar{Y}_{i,j}(\tau), i=0,1,2,$ $0\leq \tau\leq 4$ can be calculated, and it can be found $\mathbf{\Upsilon} _{i}(\tau)$ and $\mathbf{\bar{\Upsilon}} _{i}(\tau)$ are invertible. On the other hand, since $Q,\bar{Q},\mathcal{R}_{i},\bar{\mathcal{R}}_{i},\Phi(\Gamma+1),\bar{\Phi}(\Gamma+1)$ given above are positive definite, we can verify $\mathbf{\Upsilon} _{i}(\tau),\mathbf{\bar{\Upsilon}} _{i}(\tau)\geq 0$ in (\ref{equ:29}).

Thus, from Theorem \ref{theorem:2}, we know Problem \ref{problem:1} can be uniquely solved, and the optimal decentralized control is given as:
\begin{align}
v_{0}(0)=&\begin{bmatrix}
-0.2932  & -0.1466\\
   -0.1631  & -0.0815
\end{bmatrix}x(0)+\begin{bmatrix}
0.2889  &  0.2390\\
   -0.2938  & -0.0064
\end{bmatrix}\mathbb{E}x(0),\notag\\
v_{0}(1)=&\begin{bmatrix}
 0.2487  &  0.1243\\
   -0.0205  & -0.0102
\end{bmatrix}\hat{x}_{1/0}+\begin{bmatrix}
 -0.5338 &  -0.2669\\
   -0.1607  & -0.0804
\end{bmatrix}x(1)\notag\\
&+\begin{bmatrix}
0.2800  &  0.2347\\
   -0.2946  &  0.0033
\end{bmatrix}\mathbb{E}x(1),\notag\\
v_{0}(2)=&\begin{bmatrix}
0.1567  &  0.0783\\
   -0.0369 &  -0.0185
\end{bmatrix}\hat{x}_{2/0}+\begin{bmatrix}
0.0230  &  0.0115\\
    0.0097 &   0.0049
\end{bmatrix}\hat{x}_{2/1}\notag\\
&+\begin{bmatrix}
-0.4693 &  -0.2346\\
   -0.1364  & -0.0682
\end{bmatrix}x(2)+\begin{bmatrix}
0.2857  &  0.2374\\
   -0.2911  & -0.0035
\end{bmatrix}\mathbb{E}x(2),\notag\\
v_{0}(3)=&\begin{bmatrix}
 0.2185  &  0.1093\\
   -0.0297  & -0.0149
\end{bmatrix}\hat{x}_{3/1}+\begin{bmatrix}
0.0264  &  0.0132\\
    0.0111  &  0.0055
\end{bmatrix}\hat{x}_{3/2}\notag\\
&+\begin{bmatrix}
-0.5207  & -0.2603\\
   -0.1558  & -0.0779
\end{bmatrix}x(3)+\begin{bmatrix}
0.2709  &  0.2303\\
   -0.2910  &  0.0017
\end{bmatrix}\mathbb{E}x(3),\notag\\
v_{0}(4)=&\begin{bmatrix}
0.1907  &  0.0953\\
   -0.0235  & -0.0117
\end{bmatrix}\hat{x}_{4/2}+\begin{bmatrix}
 0.0175 &   0.0087\\
    0.0075  &  0.0037
\end{bmatrix}\hat{x}_{4/3}\notag\\
&+\begin{bmatrix}
-0.4499  & -0.2249\\
   -0.1291  & -0.0645
\end{bmatrix}x(4)+\begin{bmatrix}
0.2351  &  0.2105\\
   -0.2957  &  0.0034
\end{bmatrix}\mathbb{E}x(4),\notag\\
v_{1}(1)=&\begin{bmatrix}
-0.0901 &  -0.0450\\
   -0.1823  & -0.0911
\end{bmatrix}\hat{x}_{1/0}+\begin{bmatrix}
 0.0487  &  0.0265\\
    0.1279 &  -0.0750
\end{bmatrix}\mathbb{E}x(1),\notag\\
v_{1}(2)=&\begin{bmatrix}
-0.0836 &  -0.0418\\
   -0.2131  & -0.1065
\end{bmatrix}\hat{x}_{2/1}+\begin{bmatrix}
0.0004  & -0.0002\\
    0.0373  &  0.0187
\end{bmatrix}\hat{x}_{2/0}\notag\\
&+\begin{bmatrix}
0.0444  &  0.0236\\
    0.1222  & -0.0786
\end{bmatrix}\mathbb{E}x(2),\notag\\
v_{1}(3)=&\begin{bmatrix}
-0.0936 &  -0.0468\\
   -0.2365  & -0.1182
\end{bmatrix}\hat{x}_{3/2}+\begin{bmatrix}
0.0067  &  0.0034\\
    0.0606  &  0.0303
\end{bmatrix}\hat{x}_{3/1}\notag\\
&+\begin{bmatrix}
0.0463  &  0.0250\\
    0.1224  & -0.0783
\end{bmatrix}\mathbb{E}x(3),\notag\\
v_{1}(4)=&\begin{bmatrix}
 -0.0804  & -0.0402\\
   -0.2054 &  -0.1027
\end{bmatrix}\hat{x}_{4/3}+\begin{bmatrix}
- 0.0070  &  0.0035\\
    0.0546  &  0.0273
\end{bmatrix}\hat{x}_{4/2}\notag\\
&+\begin{bmatrix}
 0.0352  &  0.0200\\
    0.1036  & -0.0840
\end{bmatrix}\mathbb{E}x(4),\notag\\
v_{2}(2)=&\begin{bmatrix}
 -0.2931 &  -0.1465\\
   -0.1580 &  -0.0790
\end{bmatrix}\hat{x}_{2/0}+\begin{bmatrix}
0.4550  &  0.3004\\
    0.0456  &  0.1841
\end{bmatrix}\mathbb{E}x(2),\notag\\
v_{2}(3)=&\begin{bmatrix}
-0.3054 &  -0.1527\\
   -0.1628  & -0.0814
\end{bmatrix}\hat{x}_{3/1}+\begin{bmatrix}
 0.4702  &  0.3065\\
    0.0468  &  0.1863
\end{bmatrix}\mathbb{E}x(3),\notag\\
v_{2}(4)=&\begin{bmatrix}
-0.2646 &  -0.1323\\
   -0.1468  & -0.0734
\end{bmatrix}\hat{x}_{4/2}+\begin{bmatrix}
0.4178 &   0.2763\\
    0.0341  &  0.1784
\end{bmatrix}\mathbb{E}x(4).\notag
\end{align}

In the following, we choose
\begin{align}
x(0)=&\begin{bmatrix}
x^{1}(0)\\
x^{2}(0)
\end{bmatrix}=\begin{bmatrix}
2\\
1
\end{bmatrix},
\mathbb{E}x(0)=\begin{bmatrix}
\mathbb{E}x^{1}(0)\\
\mathbb{E}x^{2}(0)
\end{bmatrix}=\begin{bmatrix}
2\\
1
\end{bmatrix},\notag\\
\omega(\tau)&\sim\mathcal{N}(0,1),\Gamma=100.\notag
\end{align}
By using the results of Theorem \ref{theorem:2}, the state trajectory $x(\tau)=\begin{bmatrix}
x^{1}(\tau)\\
x^{2}(\tau)
\end{bmatrix}$ for $0\leq \tau\leq 100$ can be obtained, which is depicted as in FIGURE \ref{figure}. From FIGURE \ref{figure}, it is clear that the system state converges to 0.

\begin{figure}[htbp]
\centering
\includegraphics[scale=0.5]{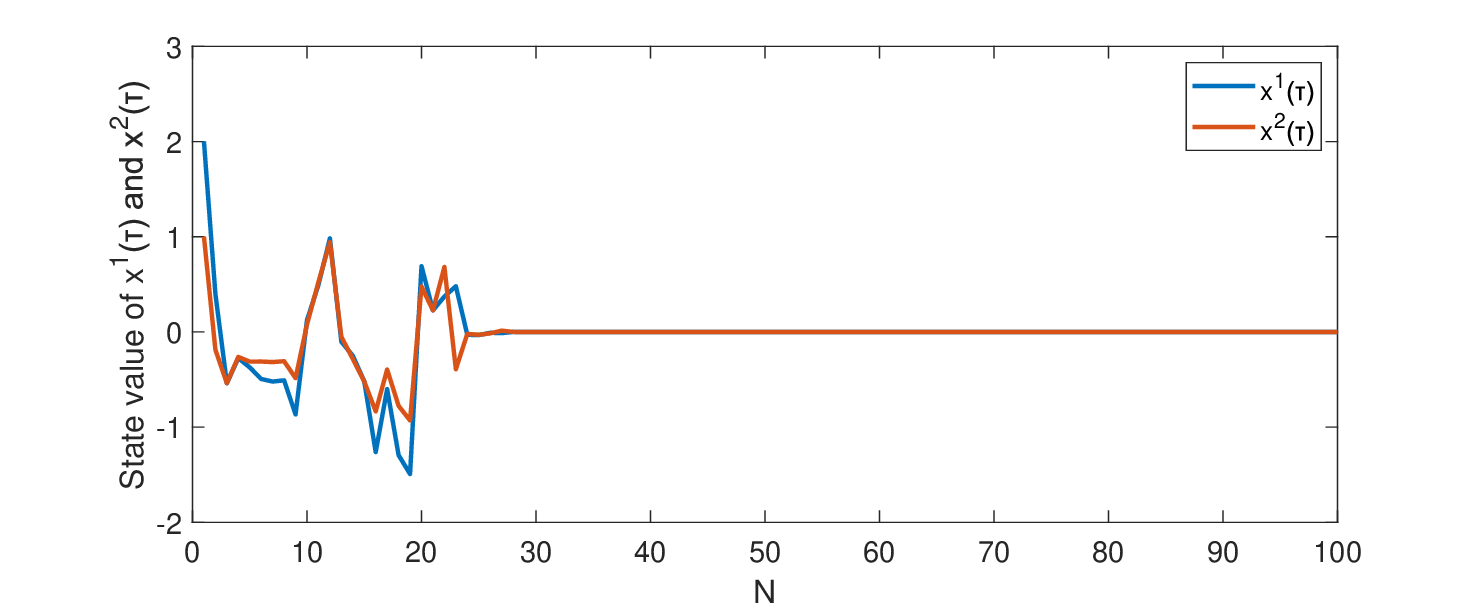}
\caption
{State trajectory $x(\tau)$ with the optimal decentralized control strategy.}
\label{figure}
\end{figure}

\section{Conclusion}

In this paper, we have investigated the decentralized control problem for the discrete-time mean field systems with multiple controllers of delayed information. Firstly, by the use of Pontryagin's  maximum principle, the necessary and sufficient solvable conditions for the decentralized control problem have been derived. Next, in order to handle with the asymmetric information pattern, the orthogonal decomposition approach has been proposed to decouple the associated MF-FBSDEs. Finally, the optimal decentralized control strategy has been derived, which is based on an asymmetric Riccati equation. For future research, we will extend the results of this paper to study the infinite time horizon case.

\section*{Acknowledgments}
This work was supported by National Natural Science Foundation of China under grant 61903210, Natural Science Foundation of Shandong Province under grant ZR2019BF002, China Postdoctoral Science Foundation under grants 2019M652324, 2021T140354, Qingdao Postdoctoral Application Research Project, Major Basic Research of Natural Science Foundation of Shandong Province under grant ZR2021ZD14.

\section*{Appendix: Proof of Theorem 1}
\begin{proof}
According to Lemma \ref{lemma:1}-Lemma \ref{lemma:2}, we know that the necessary and sufficient conditions for the unique solution of Problem \ref{problem:1} is that equilibrium condition (\ref{equ:15}) is uniquely solvable. Consequently, based on the induction method, we will solve the MF-FBSDEs \eqref{mffbsde} to find the solution of Problem \ref{problem:1}.


In the first place, for $\tau=\Gamma,$ we have $\theta(\Gamma)=\Phi(\Gamma+1)x(\Gamma+1)+\bar{\Phi}(\Gamma+1)\mathbb{E}x(\Gamma+1),$ then (\ref{equ:22}) with $i=h$ implies
\begin{align}\label{equ:34}
0=&\mathbf{R}_{h}\mathbf{V}_{h}(\Gamma)+\mathbf{\bar{R}}_{h}\mathbb{E}\mathbf{V}_{h}(\Gamma)+\mathbb{E}[\mathbf{B}_{h}^{T}(\Gamma)\theta (\Gamma)\mid \mathcal{F}_{h}(\Gamma)]+\mathbb{E}[\mathbf{\bar{B}}_{h}^{T}(\Gamma)\theta (\Gamma)]\notag\\
=&[\mathbf{B}_{h}^{T}\Phi(\Gamma+1)\mathcal{A}+\sigma ^{2}\mathbf{D}_{h}^{T}\Phi(\Gamma+1)C]\hat{x}_{\Gamma/\Gamma-h}+[\mathbf{R}_{h}+\mathbf{B}_{h}^{T}\Phi(\Gamma+1)\mathbf{B}_{h}+\sigma ^{2}\mathbf{D}_{h}^{T}\Phi(\Gamma+1)\mathbf{D}_{h}]\mathbf{V}_{h}(\Gamma)\notag\\
&+\Big[(\mathbf{B}_{h}+\mathbf{\bar{B}}_{h})^{T}[\Phi(\Gamma+1)+\bar{\Phi}(\Gamma+1)](\mathcal{A}+\bar{\mathcal{A}})-\mathbf{B}_{h}^{T}\Phi(\Gamma+1)\mathcal{A}\notag\\
&+\sigma ^{2}(\mathbf{D}_{h}+\mathbf{\bar{D}}_{h})^{T}\Phi(\Gamma+1)(C+\bar{C})-\sigma ^{2}\mathbf{D}_{h}^{T}\Phi(\Gamma+1)C\Big]\mathbb{E}x(\Gamma)\notag\\
&+\Big[\mathbf{\bar{R}}_{h}+(\mathbf{B}_{h}+\mathbf{\bar{B}}_{h})^{T}[\Phi(\Gamma+1)+\bar{\Phi}(\Gamma+1)](\mathbf{B}_{h}+\mathbf{\bar{B}}_{h})-\mathbf{B}_{h}^{T}\Phi(\Gamma+1)\mathbf{B}_{h}\notag\\
&+\sigma ^{2}(\mathbf{D}_{h}+\mathbf{\bar{D}}_{h})^{T}\Phi(\Gamma+1)(\mathbf{D}_{h}+\mathbf{\bar{D}}_{h})-\sigma ^{2}\mathbf{D}_{h}^{T}\Phi(\Gamma+1)\mathbf{D}_{h}\Big]\mathbb{E}\mathbf{V}_{h}(\Gamma)\notag\\
=&\mathbf{\bar{\Upsilon }}_{h}(\Gamma)\mathbf{V}_{h}(\Gamma)+[\mathbf{\Upsilon}_{h}(\Gamma)-\mathbf{\bar{\Upsilon }}_{h}(\Gamma)]\mathbb{E}\mathbf{V}_{h}(\Gamma)+\bar{Y}_{h,h}(\Gamma)\hat{x}_{\Gamma/\Gamma-h}+[Y_{h,h}(\Gamma)-\bar{Y}_{h,h}(\Gamma)]\mathbb{E}x(\Gamma).
\end{align}
By taking the mathematical expectation of the (\ref{equ:34}), we have
\begin{align}\label{equ:35}
0=&\Big[\mathbf{R}_{h}+\mathbf{\bar{R}}_{h}+(\mathbf{B}_{h}+\mathbf{\bar{B}}_{h})^{T}[\Phi(\Gamma+1)+\bar{\Phi}(\Gamma+1)](\mathbf{B}_{h}+\mathbf{\bar{B}}_{h})+\sigma ^{2}(\mathbf{D}_{h}+\mathbf{\bar{D}}_{h})^{T}\Phi(\Gamma+1)(\mathbf{D}_{h}+\mathbf{\bar{D}}_{h})\Big]\mathbb{E}\mathbf{V}_{h}(\Gamma)\notag\\
&+\Big[(\mathbf{B}_{h}+\mathbf{\bar{B}}_{h})^{T}[\Phi(\Gamma+1)+\bar{\Phi}(\Gamma+1)](\mathcal{A}+\bar{\mathcal{A}})+\sigma ^{2}(\mathbf{D}_{h}+\mathbf{\bar{D}}_{h})^{T}\Phi(\Gamma+1)(C+\bar{C})\big]\mathbb{E}x(\Gamma)\notag\\
=&\mathbf{\Upsilon} _{h}(\Gamma)\mathbb{E}\mathbf{V}_{h}(\Gamma)+Y_{h,h}(\Gamma)\mathbb{E}x(\Gamma).
\end{align}
{\color{blue}Further, $\mathbb{E}\mathbf{V}_{h}(\Gamma)$ can be calculated as}
\begin{align}\label{equ:36}
\mathbb{E}\mathbf{V}_{h}(\Gamma)=-\mathbf{\Upsilon}_{h}^{-1}(\Gamma)Y_{h,h}(\Gamma)\mathbb{E}x(\Gamma).
\end{align}
Substitute (\ref{equ:36}) into (\ref{equ:34}), we have
\begin{align}\label{equ:37}
0=&\mathbf{\bar{\Upsilon }}_{h}(\Gamma)\mathbf{V}_{h}(\Gamma)+[\mathbf{\Upsilon}_{h}(\Gamma)-\mathbf{\bar{\Upsilon }}_{h}(\Gamma)]\mathbb{E}\mathbf{V}_{h}(\Gamma)\notag\\
&+\bar{Y}_{h,h}(\Gamma)\hat{x}_{\Gamma/\Gamma-h}+[Y_{h,h}(\Gamma)-\bar{Y}_{h,h}(\Gamma)]\mathbb{E}x(\Gamma)\notag\\
=&\mathbf{\bar{\Upsilon }}_{h}(\Gamma)\mathbf{V}_{h}(\Gamma)-[\mathbf{\Upsilon}_{h}(\Gamma)-\mathbf{\bar{\Upsilon }}_{h}(\Gamma)]\mathbf{\Upsilon}_{h}^{-1}(\Gamma)Y_{h,h}(\Gamma)\mathbb{E}x(\Gamma)\notag\\
&+\bar{Y}_{h,h}(\Gamma)\hat{x}_{\Gamma/\Gamma-h}+[Y_{h,h}(\Gamma)-\bar{Y}_{h,h}(\Gamma)]\mathbb{E}x(\Gamma)\notag\\
=&\mathbf{\bar{\Upsilon }}_{h}(\Gamma)\mathbf{V}_{h}(\Gamma)+\bar{Y}_{h,h}(\Gamma)\hat{x}_{\Gamma/\Gamma-h}+[\mathbf{\bar{\Upsilon }}_{h}(\Gamma)\mathbf{\Upsilon}_{h}^{-1}(\Gamma)Y_{h,h}(\Gamma)-\bar{Y}_{h,h}(\Gamma)]\mathbb{E}x(\Gamma).
\end{align}

According to Lemma \ref{lemma:1}, it is noted that the equilibrium condition (\ref{equ:15}) for $\tau=\Gamma$ is uniquely solvable if and only if $\mathbf{\Upsilon}_{h}(\Gamma)$ and $\mathbf{\bar{\Upsilon}}_{h}(\Gamma)$ are invertible.
The unique optimal control $v_{h}(\Gamma)$ can be obtained as
\begin{align}\label{equ:39}
v_{h}(\Gamma)=&\mathbf{I}_{h}\mathbf{V}_{h}(\Gamma)=\mathbf{I}_{h}\begin{bmatrix}
\hat{v}_{0,h}(\Gamma)\\
\vdots\\
\hat{v}_{h-1,h}(\Gamma)\\
v_{h}(\Gamma)
\end{bmatrix}\notag\\
=&-\mathbf{I}_{h}\mathbf{\bar{\Upsilon }}_{h}^{-1}(\Gamma)\bar{Y}_{h,h}(\Gamma)\hat{x}_{\Gamma/\Gamma-h}-\mathbf{I}_{h}[\mathbf{\Upsilon}_{h}^{-1}(\Gamma)Y_{h,h}(\Gamma)-\mathbf{\bar{\Upsilon }}_{h}^{-1}(\Gamma)\bar{Y}_{h,h}(\Gamma)]\mathbb{E}x(\Gamma).
\end{align}

Next, we will calculate $V_{h-1}(\Gamma).$ In fact, $x(\Gamma+1)$ can be rewritten as
\begin{align}\label{equ:40}
x(\Gamma+1)=&[\mathcal{A}+\omega(\Gamma)C]x(\Gamma)+[\bar{\mathcal{A}}+\omega(\Gamma)\bar{C}]\mathbb{E}x(\Gamma)\notag\\
&+[\mathbf{B}_{h-1}+\omega(\Gamma)\mathbf{D}_{h-1}]\mathbf{V}_{h-1}(\Gamma)+[\mathbf{\bar{B}}_{h-1}+\omega(\Gamma)\mathbf{\bar{D}}_{h-1}] \mathbb{E}\mathbf{V}_{h-1}(\Gamma)\notag\\
&+\sum_{j=0}^{h-2}\Big[[\mathcal{B}_{j}+\omega(\Gamma)D_{j}]\tilde{v}_{j,i}(\tau)+[\bar{\mathcal{B}}_{j}+\omega(\Gamma)\bar{D}]\mathbb{E}\tilde{v}_{j,i}(\tau)\Big]\notag\\
&+[\mathcal{B}_{h}+\omega(\Gamma)D_{h}]v_{h}(\Gamma)+[\bar{\mathcal{B}}_{h}+\omega(\Gamma)\bar{D}_{h}]\mathbb{E}v_{h}(\Gamma).
\end{align}

Then, applying (\ref{equ:40}) to (\ref{equ:22}) for $i=h-1$, we can obtain
\begin{align}\label{equ:41}
0=&\mathbf{R}_{h-1}\mathbf{V}_{h-1}(\Gamma)+\mathbf{\bar{R}}_{h-1}\mathbb{E}\mathbf{V}_{h-1}(\Gamma)+\mathbb{E}[\mathbf{B}_{h-1}^{T}(\Gamma)\theta (\Gamma)\mid \mathcal{F}_{h-1}(\Gamma)]+\mathbb{E}[\mathbf{\bar{B}}_{h-1}^{T}(\Gamma)\theta (\Gamma)]\notag\\
=&\mathbf{\bar{\Upsilon }}_{h-1}(\Gamma)\mathbf{V}_{h-1}(\Gamma)+[\mathbf{\Upsilon}_{h-1}(\Gamma)-\mathbf{\bar{\Upsilon }}_{h-1}(\Gamma)]\mathbb{E}\mathbf{V}_{h-1}(\Gamma)\notag\\
&+\bar{Y}_{h-1,h-1}(\Gamma)\hat{x}_{\Gamma/\Gamma-h+1}+[Y_{h-1,h-1}(\Gamma)-\bar{Y}_{h-1,h-1}(\Gamma)]\mathbb{E}x(\Gamma)\notag\\
&+[\mathbf{B}_{h-1}^{T}\Phi(\Gamma+1)\mathcal{B}_{h}+\sigma ^{2}\mathbf{D}_{h-1}^{T}\Phi(\Gamma+1)D_{h}]v_{h}(\Gamma)\notag\\
&+\Big[(\mathbf{B}_{h-1}+\mathbf{\bar{B}}_{h-1})^{T}[\Phi(\Gamma+1)+\bar{\Phi}(\Gamma+1)](\mathcal{B}_{h}+\bar{\mathcal{B}}_{h})-\mathbf{B}_{h-1}^{T}\Phi(\Gamma+1)\mathcal{B}_{h}\notag\\
&+\sigma ^{2}(\mathbf{D}_{h-1}+\mathbf{\bar{D}}_{h-1})^{T}\Phi(\Gamma+1)(D_{h}+\bar{D}_{h})-\sigma ^{2}\mathbf{D}_{h-1}^{T}\Phi(\Gamma+1)D_{h}\Big]\mathbb{E}v_{h}(\Gamma)\notag\\
=&\mathbf{\bar{\Upsilon }}_{h-1}(\Gamma)\mathbf{V}_{h-1}(\Gamma)+[\mathbf{\Upsilon}_{h-1}(\Gamma)-\mathbf{\bar{\Upsilon }}_{h-1}(\Gamma)]\mathbb{E}\mathbf{V}_{h-1}(\Gamma)\notag\\
&+\bar{Y}_{h-1,h-1}(\Gamma)\hat{x}_{\Gamma/\Gamma-h+1}+\bar{Y}_{h-1,h}(\Gamma)\hat{x}_{\Gamma/\Gamma-h}\notag\\
&+\Big[[Y_{h-1,h-1}(\Gamma)-\bar{Y}_{h-1,h-1}(\Gamma)]+[Y_{h-1,h}(\Gamma)-\bar{Y}_{h-1,h}(\Gamma)]\Big]\mathbb{E}x(\Gamma).
\end{align}

Similarly, by taking the mathematical expectation of the (\ref{equ:41}), there holds
\begin{align}\label{equ:42}
0=&\mathbf{\Upsilon} _{h-1}(\Gamma)\mathbb{E}\mathbf{V}_{h-1}(\Gamma)+[Y_{h-1,h-1}(\Gamma)+Y_{h-1,h}(\Gamma)]\mathbb{E}x(\Gamma),
\end{align}
it yields that
\begin{align}\label{equ:43}
\mathbb{E}\mathbf{V}_{h-1}(\Gamma)=&-\mathbf{\Upsilon} _{h-1}^{-1}(\Gamma)[Y_{h-1,h-1}(\Gamma)+Y_{h-1,h}(\Gamma)]\mathbb{E}x(\Gamma).
\end{align}
Substitute (\ref{equ:43}) into (\ref{equ:41}), we have
\begin{align}\label{equ:44}
0=&\mathbf{\bar{\Upsilon }}_{h-1}(\Gamma)\mathbf{V}_{h-1}(\Gamma)+[\mathbf{\Upsilon}_{h-1}(\Gamma)-\mathbf{\bar{\Upsilon }}_{h-1}(\Gamma)]\mathbb{E}\mathbf{V}_{h-1}(\Gamma)\notag\\
&+\bar{Y}_{h-1,h-1}(\Gamma)\hat{x}_{\Gamma/\Gamma-h+1}+\bar{Y}_{h-1,h}(\Gamma)\hat{x}_{\Gamma/\Gamma-h}\notag\\
&+\Big[[Y_{h-1,h-1}(\Gamma)-\bar{Y}_{h-1,h-1}(\Gamma)]+[Y_{h-1,h}(\Gamma)-\bar{Y}_{h-1,h}(\Gamma)]\Big]\mathbb{E}x(\Gamma)\notag\\
=&\mathbf{\bar{\Upsilon }}_{h-1}(\Gamma)\mathbf{V}_{h-1}(\Gamma)-[\mathbf{\Upsilon}_{h-1}(\Gamma)-\mathbf{\bar{\Upsilon }}_{h-1}(\Gamma)]\mathbf{\Upsilon} _{h-1}^{-1}(\Gamma)[Y_{h-1,h-1}(\Gamma)\notag\\
&+Y_{h-1,h}(\Gamma)]\mathbb{E}x(\Gamma)+\bar{Y}_{h-1,h-1}(\Gamma)\hat{x}_{\Gamma/\Gamma-h+1}+\bar{Y}_{h-1,h}(\Gamma)\hat{x}_{\Gamma/\Gamma-h}\notag\\
&+\Big[[Y_{h-1,h-1}(\Gamma)-\bar{Y}_{h-1,h-1}(\Gamma)]+[Y_{h-1,h}(\Gamma)-\bar{Y}_{h-1,h}(\Gamma)]\Big]\mathbb{E}x(\Gamma)\notag\\
=&\mathbf{\bar{\Upsilon }}_{h-1}(\Gamma)\mathbf{V}_{h-1}(\Gamma)+\bar{Y}_{h-1,h-1}(\Gamma)\hat{x}_{\Gamma/\Gamma-h+1}+\bar{Y}_{h-1,h}(\Gamma)\hat{x}_{\Gamma/\Gamma-h}\notag\\
&+\Big[\mathbf{\bar{\Upsilon }}_{h-1}(\Gamma)\mathbf{\Upsilon}_{h-1}^{-1}(\Gamma)[Y_{h-1,h-1}(\Gamma)+Y_{h-1,h}(\Gamma)]-[\bar{Y}_{h-1,h-1}(\Gamma)+\bar{Y}_{h-1,h}(\Gamma)]\Big]\mathbb{E}x(\Gamma).
\end{align}

Similar to the analysis above (\ref{equ:34}), (\ref{equ:15}) is uniquely solvable if and only if $\mathbf{\Upsilon}_{h-1}(\Gamma)$ and $\mathbf{\bar{\Upsilon}}_{h-1}(\Gamma)$ are invertible. Apparently, $v_{h-1}(\Gamma)$ can be derived.

The remainder of the argument is analogous to the above analysis, so by repeating the above procedures step by step, it can be verified that $\mathbf{V}_{i}(\Gamma)$ can be given by
\begin{align}\label{equ:45}
\mathbf{V}_{i}(\Gamma)=&-\mathbf{\bar{\Upsilon}} _{i}^{-1}(\Gamma)\sum_{j=i}^{h}\bar{Y}_{i,j}(\Gamma)\hat{x}_{\Gamma/\Gamma-j}-\sum_{j=i}^{h}[\mathbf{\Upsilon} _{i}^{-1}(\Gamma)Y_{i,j}(\Gamma)-\mathbf{\bar{\Upsilon}} _{i}^{-1}(\Gamma)\bar{Y}_{i,j}(\Gamma)]\mathbb{E}x(\Gamma),
\end{align}
and $v_{i}(\Gamma)$ is given by
\begin{align}\label{equ:46}
v_{i}(\Gamma)=&-\mathbf{I}_{i}\mathbf{\bar{\Upsilon}} _{i}^{-1}(\Gamma)\sum_{j=i}^{h}\bar{Y}_{i,j}(\Gamma)\hat{x}_{\Gamma/\Gamma-j}-\sum_{j=i}^{h}\mathbf{I}_{i}[\mathbf{\Upsilon} _{i}^{-1}(\Gamma)Y_{i,j}(\Gamma)-\mathbf{\bar{\Upsilon}} _{i}^{-1}(\Gamma)\bar{Y}_{i,j}(\Gamma)]\mathbb{E}x(\Gamma),
\end{align}
where $Y_{i,i}(\Gamma)$,$\bar{Y}_{i,i}(\Gamma)$,$Y_{i,j}(\Gamma)$,$\bar{Y}_{i,j}(\Gamma)$ satisfy the following relationships
\begin{align}\label{equ:60}
\left\{
\begin{aligned}
\bar{Y}_{i,i}(\Gamma)=&\mathbf{B}_{i}^{T}\Phi(\Gamma+1)\mathcal{A}+\sigma ^{2}\mathbf{D}_{i}^{T}\Phi(\Gamma+1)C,\\
Y_{i,i}(\Gamma)=&(\mathbf{B}_{i}+\mathbf{\bar{B}}_{i})^{T}[\Phi(\Gamma+1)+\bar{\Phi}(\Gamma+1)](\mathcal{A}+\bar{\mathcal{A}})+\sigma ^{2}(\mathbf{D}_{i}+\mathbf{\bar{D}}_{i})^{T}\Phi(\Gamma+1)(C+\bar{C}),\\
\bar{Y}_{i,j}(\Gamma)=&-[\mathbf{B}_{i}^{T}\Phi(\Gamma+1)\mathcal{B}_{j}+\sigma ^{2}\mathbf{D}_{i}^{T}\Phi(\Gamma+1)D_{j}]\mathbf{I}_{j}\mathbf{\bar{\Upsilon}}_{j}^{-1}\sum_{m=j}^{h}\bar{Y}_{j,m}(\Gamma),\\
Y_{i,j}(\Gamma)=&-[(\mathbf{B}_{i}+\mathbf{\bar{B}}_{i})^{T}[\Phi(\Gamma+1)+\bar{\Phi}(\Gamma+1)](\mathcal{B}_{j}+\bar{\mathcal{B}}_{j})\\
&+\sigma ^{2}(\mathbf{D}_{i}+\mathbf{\bar{D}}_{i})^{T}\Phi(\Gamma+1)(D_{j}+\bar{D}_{j})]\mathbf{I}_{j}\mathbf{\Upsilon}_{j}^{-1}\sum_{m=j}^{h}Y_{j,m}(\Gamma),i+1\leq j\leq h.\\
\end{aligned}
\right.
\end{align}
In other words, $v_i(\Gamma)$ in \eqref{equ:30} has been verified for $i=0,\cdots,h$.

Afterwards, we will concentrate on calculating $\theta(\Gamma-1)$. From (\ref{equ:11}), we can obtain that
\begin{align}\label{equ:47}
\theta (\Gamma-1)=&Qx(\Gamma)+\bar{Q}\mathbb{E}x(\Gamma)+\mathbb{E}[\mathcal{A}^{T}(\Gamma)\theta (\Gamma)\mid \mathcal{F}_{0}(\Gamma)]+\mathbb{E}[\bar{\mathcal{A}}^{T}(\Gamma)\theta (\Gamma)]\notag\\
=&\big[Q+\mathcal{A}^{T}\Phi(\Gamma+1)\mathcal{A}+\sigma ^{2}C^{T}\Phi(\Gamma+1)C\big]x(\Gamma)\notag\\
&+\Big[\bar{Q}+(\mathcal{A}+\bar{\mathcal{A}})^{T}[\Phi(\Gamma+1)+\bar{\Phi}(\Gamma+1)](\mathcal{A}+\bar{\mathcal{A}})-\mathcal{A}^{T}\Phi(\Gamma+1)\mathcal{A}+\sigma ^{2}(C+\bar{C})^{T}\Phi(\Gamma+1)(C+\bar{C})\notag\\
&-\sigma ^{2}C^{T}\Phi(\Gamma+1)C+\sum_{i=0}^{h}[\bar{L}_{i}^{T}(\Gamma)\mathbf{I}_{i}\mathbf{\bar{\Upsilon}} _{i}^{-1}(\Gamma)\sum_{j=i}^{h}\bar{Y}_{i,j}(\Gamma)-L_{i}^{T}(\Gamma)\mathbf{I}_{i}\mathbf{\Upsilon} _{i}^{-1}(\Gamma)\sum_{j=i}^{h}Y_{i,j}(\Gamma)]\Big]\mathbb{E}x(\Gamma)\notag\\
&-\sum_{i=0}^{h}\bar{L}_{i}^{T}(\Gamma)\mathbf{I}_{i}\mathbf{\bar{\Upsilon}} _{i}^{-1}(\Gamma)\sum_{j=i}^{h}\bar{Y}_{i,j}(\Gamma)\hat{x}_{\Gamma/\Gamma-j}\notag\\
=&\Phi(\Gamma)x(\Gamma)+\bar{\Phi}(\Gamma)\mathbb{E}x(\Gamma)+\sum_{j=1}^{h}\varphi_{j}(\Gamma)\hat{x}_{\Gamma/\Gamma-j},
\end{align}
where we have used the relationships
\begin{align}\label{equ:48}
\sum_{i=0}^{h}[\bar{L}_{i}^{T}(\Gamma)&\mathbf{I}_{i}\mathbf{\bar{\Upsilon}} _{i}^{-1}(\Gamma)\sum_{j=i}^{h}\bar{Y}_{i,j}(\Gamma)-L_{i}^{T}(\Gamma)\mathbf{I}_{i}\mathbf{\Upsilon} _{i}^{-1}(\Gamma)\sum_{j=i}^{h}Y_{i,j}(\Gamma)]\notag\\
=\sum_{j=0}^{h}\Big[\sum_{i=0}^{j}[&\bar{L}_{i}^{T}(\Gamma)\mathbf{I}_{i}\mathbf{\bar{\Upsilon}} _{i}^{-1}(\Gamma)\bar{Y}_{i,j}(\Gamma)-L_{i}^{T}(\Gamma)\mathbf{I}_{i}\mathbf{\Upsilon} _{i}^{-1}(\Gamma)Y_{i,j}(\Gamma)]\Big],
\end{align}
and
\begin{align}\label{equ:49}
\sum_{i=0}^{h}\bar{L}_{i}^{T}(\Gamma)\mathbf{I}_{i}\mathbf{\bar{\Upsilon}} _{i}^{-1}&(\Gamma)\sum_{j=i}^{h}\bar{Y}_{i,j}(\Gamma)=\sum_{j=0}^{h}[\sum_{i=0}^{j}\bar{L}_{i}^{T}(\Gamma)\mathbf{I}_{i}\mathbf{\bar{\Upsilon}} _{i}^{-1}(\Gamma)\bar{Y}_{i,j}(\Gamma)],
\end{align}
which mean that \eqref{equ:32} has been proved for $\tau=\Gamma-1$.

In order to adopt the induction method, we assume that the assertions in Theorem \ref{theorem:2} also hold when $\xi+1\leq \tau\leq \Gamma$. In other words, for $\xi+1\leq \tau\leq \Gamma$, it is assumed that the following assertions hold:
\begin{itemize}
  \item [1)](\ref{equ:15}) is uniquely solvable if and only if $\mathbf{\Upsilon}_{i}(\tau)$ and $\mathbf{\bar{\Upsilon}}_{i}(\tau)$ are invertible, and the control strategy $v_{i}(\tau)$ satisfies (\ref{equ:30});
  \item [2)]The Riccati equations (\ref{equ:101}) hold;
  \item [3)]$\theta(\tau-1)$ is given by (\ref{equ:32}).
\end{itemize}

In the following, we will verify (\ref{equ:30}) for $\tau=\xi$. In fact, from (\ref{equ:22}), we have
\begin{align}\label{equ:50}
0=&\mathbf{R}_{h}\mathbf{V}_{h}(\xi)+\mathbf{\bar{R}}_{h}\mathbb{E}\mathbf{V}_{h}(\xi)+\mathbb{E}[\mathbf{B}_{h}^{T}(\xi)\theta (\xi)\mid \mathcal{F}_{h}(\xi)]+\mathbb{E}[\mathbf{\bar{B}}_{h}^{T}(\xi)\theta(\xi)]\notag\\
=&\mathbf{R}_{h}\mathbf{V}_{h}(\xi)+\mathbf{\bar{R}}_{h}\mathbb{E}\mathbf{V}_{h}(\xi)+\mathbb{E}\Big[[\mathbf{B}_{h}+\omega(\xi)\mathbf{D}_{h}]^{T}[\Phi(\xi+1)x(\xi+1)\notag\\
&+\bar{\Phi}(\xi+1)\mathbb{E}x(\xi+1)+\sum_{j=1}^{h}\varphi_{j}(\xi+1)\hat{x}_{\xi+1/\xi+1-j}]\mid \mathcal{F}_{h}(\xi)\Big]+\mathbb{E}\Big[[\mathbf{\bar{B}}_{h}+\omega(\xi)\mathbf{\bar{D}}_{h}]^{T}\notag\\
&\times[\Phi(\xi+1)x(\xi+1)+\bar{\Phi}(\xi+1)\mathbb{E}x(\xi+1)+\sum_{j=1}^{h}\varphi_{j}(\xi+1)\hat{x}_{\xi+1/\xi+1-j}]\Big]\notag\\
=&\mathbf{\bar{\Upsilon }}_{h}(\xi)\mathbf{V}_{h}(\xi)+[\mathbf{\Upsilon}_{h}(\xi)-\mathbf{\bar{\Upsilon }}_{h}(\xi)]\mathbb{E}\mathbf{V}_{h}(\xi)+\bar{Y}_{h,h}(\xi)\hat{x}_{\xi/\xi-h}+[Y_{h,h}(\xi)-\bar{Y}_{h,h}(\xi)]\mathbb{E}x(\xi).
\end{align}
Hence, (\ref{equ:15}) is uniquely solvable if and only if $\mathbf{\Upsilon}_{h}(\xi)$ and $\mathbf{\bar{\Upsilon}}_{h}(\xi)$ are invertible. In this case, $\mathbf{V}_{h}(\xi)$ and $v_{h}(\xi)$ in  (\ref{equ:30}) can be calculated from (\ref{equ:50}), respectively.

In the following, we will verify $v_{h-1}(\xi)$. Similar to the arguments in (\ref{equ:41}), it follows that

\begin{align}\label{equ:51}
0=&\mathbf{R}_{h-1}\mathbf{V}_{h-1}(\xi)+\mathbf{\bar{R}}_{h-1}\mathbb{E}\mathbf{V}_{h-1}(\xi)+\mathbb{E}[\mathbf{B}_{h-1}^{T}(\xi)\theta (\xi)\mid \mathcal{F}_{h-1}(\xi)]+\mathbb{E}[\mathbf{\bar{B}}^{T}_{h-1}(\xi)\theta(\xi)]\notag\\
=&\mathbf{R}_{h-1}\mathbf{V}_{h-1}(\xi)+\mathbf{\bar{R}}_{h-1}\mathbb{E}\mathbf{V}_{h-1}(\xi)+\mathbb{E}\Big[[\mathbf{B}_{h-1}+\omega(\xi)\mathbf{D}_{h-1}]^{T}[\Phi(\xi+1)x(\xi+1)\notag\\
&+\bar{\Phi}(\xi+1)\mathbb{E}x(\xi+1)+\sum_{j=1}^{h}\varphi_{j}(\xi+1)\hat{x}_{\xi+1/\xi+1-j}]\mid \mathcal{F}_{h-1}(\xi)\Big]\notag\\
&+\mathbb{E}\Big[[\mathbf{\bar{B}}_{h-1}+\omega(\xi)\mathbf{\bar{D}}_{h-1}]^{T}[\Phi(\xi+1)x(\xi+1)+\bar{\Phi}(\xi+1)\mathbb{E}x(\xi+1)+\sum_{j=1}^{h}\varphi_{j}(\xi+1)\hat{x}_{\xi+1/\xi+1-j}]\Big]\notag\\
=&\mathbf{\bar{\Upsilon}}_{h-1}(\xi)\mathbf{V}_{h-1}(\xi)+[\mathbf{\Upsilon}_{h-1}(\xi)-\mathbf{\bar{\Upsilon}}_{h-1}(\xi)]\mathbb{E}\mathbf{V}_{h-1}+\bar{Y}_{h-1,h-1}(\xi)\hat{x}_{\xi/\xi-h+1}+\bar{Y}_{h-1,h}(\xi)\hat{x}_{\xi/\xi-h}\notag\\
&+\Big[[Y_{h-1,h-1}(\xi)-\bar{Y}_{h-1,h-1}(\xi)]+[Y_{h-1,h}(\xi)-\bar{Y}_{h-1,h}(\xi)]\Big]\mathbb{E}x(\xi),
\end{align}
i.e., $v_{h-1}(\xi)$ can be proved from (\ref{equ:51}).

Following the derivations of (\ref{equ:41})-(\ref{equ:60}) and (\ref{equ:50}), it can be verified that $\mathbf{V}_{i}(\xi)$ satisfies (\ref{equ:30}), and $v_{i}(\xi)$ can be expressed as (\ref{equ:30}).

Subsequently, to end the induction method, $\theta(\xi-1)$ will be calculated.
\begin{align}\label{equ:52}
\theta (\xi&-1)=Qx(\xi)+\bar{Q}\mathbb{E}x(\xi)+\mathbb{E}[\mathcal{A}^{T}(\xi)\theta (\xi)\mid \mathcal{F}_{0}(\xi)]+\mathbb{E}[\bar{\mathcal{A}}^{T}(\xi)\theta (\xi)]\notag\\
=&Qx(\xi)+\bar{Q}\mathbb{E}x(\xi)+\mathbb{E}\Big[[\mathcal{A}+\omega (\xi)C]^{T}\big[\Phi(\xi+1)x(\xi+1)+\bar{\Phi}(\xi+1)\mathbb{E}x(\xi+1)+\sum_{j=1}^{h}\varphi_{j}(\xi+1)\hat{x}_{\xi+1/\xi+1-j}\big]\mid \mathcal{F}_{0}(\xi)\Big]\notag\\
&+\mathbb{E}\Big[[\bar{\mathcal{A}}+\omega (\xi)\bar{C}]^{T} \big[\Phi(\xi+1)x(\xi+1)+\bar{\Phi}(\xi+1)\mathbb{E}x(\xi+1)+\sum_{j=1}^{h}\varphi_{j}(\xi+1)\hat{x}_{\xi+1/\xi+1-j}\big]\Big]\notag\\
=&\Big[Q+\mathcal{A}^{T}\Phi(\xi+1)\mathcal{A}+\sigma ^{2}C^{T}\Phi(\xi+1)C+\mathcal{A}^{T}\varphi_{1}(\xi+1)\mathcal{A}-[\bar{L}_{0}^{T}(\xi)+\mathcal{A}^{T}\varphi_{1}(\xi+1)\mathbf{B}_{0}]\mathbf{\bar{\Upsilon}}_{0}^{-1}(\xi)\bar{Y}_{0,0}(\xi)\Big]x(\xi)\notag\\
&+\bigg[\bar{Q}+(\mathcal{A}+\bar{\mathcal{A}})^{T}[\Phi(\xi+1)+\bar{\Phi}(\xi+1)+\sum_{j=1}^{h}\varphi_{j}(\xi+1)](\mathcal{A}+\bar{\mathcal{A}})-\mathcal{A}^{T}[\Phi(\xi+1)+\sum_{j=1}^{h}\varphi_{j}(\xi+1)]\mathcal{A}\notag\\
&+\sigma ^{2}(C+\bar{C})^{T}\Phi(\xi+1)(C+\bar{C})-\sigma ^{2}C^{T}\Phi(\xi+1)C-\sum_{i=0}^{h}\Big[[L_{i}^{T}(\xi)\mathbf{I}_{i}+(\mathcal{A}+\bar{\mathcal{A}})^{T}\varphi_{i+1}(\xi+1)(\mathbf{B}_{i}+\mathbf{\bar{B}}_{i})]\notag\\
&\times\sum_{j=i}^{h}\mathbf{\Upsilon} _{i}^{-1}(\xi)Y_{i,j}(\xi)-[\bar{L}_{i}^{T}(\xi)\mathbf{I}_{i}+\mathcal{A}^{T}\varphi_{i+1}(\xi+1)\mathbf{B}_{i}]\sum_{j=i}^{h}\mathbf{\bar{\Upsilon}} _{i}^{-1}(\xi)\bar{Y}_{i,j}(\xi)\Big]\bigg]\mathbb{E}x(\xi)\notag\\
&-\sum_{j=1}^{h}\Big[\sum_{i=0}^{j}[\bar{L}_{i}^{T}(\xi)\mathbf{I}_{i}+\mathcal{A}^{T}\varphi_{i+1}(\xi+1)\mathbf{B}_{i}]\mathbf{\bar{\Upsilon}} _{i}^{-1}(\xi)\bar{Y}_{i,j}(\xi)-\mathcal{A}^{T}\varphi_{j+1}(\xi+1)\mathcal{A}\Big]\hat{x}_{\xi/\xi-j}\notag\\
=&\Phi(\xi)x(\xi)+\bar{\Phi}(\xi)\mathbb{E}x(\xi)+\sum_{j=1}^{h}\varphi_{j}(\xi)\hat{x}_{\xi/\xi-j},
\end{align}
which leads to (\ref{equ:32}). This ends the induction method.

As a consequence, we have shown that the unique solvability of (\ref{equ:15}) is equivalent to the invertibility of $\mathbf{\Upsilon}_{i}(\tau)$ and $\mathbf{\bar{\Upsilon}}_{i}(\tau)$. In other words, Problem \ref{problem:1} is uniquely solvable if and only if $\mathbf{\Upsilon}_{i}(\tau)$ and $\mathbf{\bar{\Upsilon}}_{i}(\tau)$ are invertible for $0\leq i\leq h,i\leq \tau\leq \Gamma.$
Furthermore, the unique optimal control strategy can be presented as (\ref{equ:30}).

Finally, we shall calculate the optimal cost functional with the optimal control $v_{i}(\tau)$ in (\ref{equ:30}). Actually, it follows that:
\begin{align}\label{equ:53}
\mathbb{E}[x^{T}&(\tau)\theta(\tau-1)-x^{T}(\tau+1)\theta(\tau)]\notag\\
=&E\bigg[x^{T}(\tau)\Big[Qx(\tau)+\bar{Q}\mathbb{E}x(\tau)+\mathbb{E}[\mathcal{A}^{T}(\tau)\theta(\tau)\mid \mathcal{F}_{0}(\tau)]+\mathbb{E}[\bar{\mathcal{A}}^{T}(\tau)\theta(\tau)]\Big]\notag\\
&-\big\{\mathcal{A}(\tau)x(\tau)+\bar{\mathcal{A}}(\tau)\mathbb{E}x(\tau)+\sum_{i=0}^{h}[\mathcal{B}_{i}(\tau)v_{i}(\tau)+\bar{\mathcal{B}}_{i}(\tau)\mathbb{E}v_{i}(\tau)]\big\}^{T}\theta(\tau)\bigg]\notag\\
=&\mathbb{E}\bigg[x^{T}(\tau)Qx(\tau)+[\mathbb{E}x(\tau)]^{T}\bar{Q}\mathbb{E}x(\tau)-\sum_{i=0}^{h}v_{i}^{T}(\tau)\Big[\mathbb{E}[\mathcal{B}_{i}^{T}(\tau)\theta(\tau)\mid \mathcal{F}_{i}(\tau)]+\mathbb{E}[\bar{\mathcal{B}}_{i}^{T}(\tau)\theta(\tau)]\Big]\bigg]\notag\\
=&\mathbb{E}\bigg[x^{T}(\tau)Qx(\tau)+[\mathbb{E}x(\tau)]^{T}\bar{Q}\mathbb{E}x(\tau)+\sum_{i=0}^{h}\Big[v_{i}^T(\tau)\mathcal{R}_{i}v_{i}(\tau)+[\mathbb{E}v_{i}(\tau)]^T\bar{\mathcal{R}}_{i}\mathbb{E}v_{i}(\tau)\Big]\bigg].
\end{align}

Then adding from $0$ to $\Gamma$ on both sides of (\ref{equ:53}), there holds that
\begin{align}\label{equ:54}
\mathbb{E}[x^{T}&(0)\theta(-1)-x^{T}(\Gamma+1)\theta(\Gamma)]\notag\\
=&E\Big[x^{T}(0)\theta(-1)-x^{T}(\Gamma+1)[\Phi(\Gamma+1)x(\Gamma+1)+\bar{\Phi}(\Gamma+1)\mathbb{E}x(\Gamma+1)]\Big]\notag\\
=&\sum_{\tau=0}^{\Gamma}\mathbb{E}\Big[x^T(\tau)Qx(\tau)+[\mathbb{E}x(\tau)]^T\bar{Q}\mathbb{E}x(\tau)\Big]+\sum_{i=0}^{h}\sum_{\tau=i}^{\Gamma}\Big[v_{i}^T(\tau)\mathcal{R}_{i}v_{i}(\tau)+[\mathbb{E}v_{i}(\tau)]^T\bar{\mathcal{R}}_{i}\mathbb{E}v_{i}(\tau)\Big].
\end{align}
Moreover, (\ref{equ:4}) can be rewritten as
\begin{align} J_{\Gamma}(v)=&\sum_{\tau=0}^{i-1}\mathbb{E}\Big[x^T(\tau)Qx(\tau)+[\mathbb{E}x(\tau)]^T\bar{Q}\mathbb{E}x(\tau)+\sum_{i=0}^{h}\big[v_{i}^T(\tau)\mathcal{R}_{i}v_{i}(\tau)\notag\\
&+[\mathbb{E}v_{i}(\tau)]^T\bar{\mathcal{R}}_{i}\mathbb{E}v_{i}(\tau)\big]\Big]+\sum_{\tau=i}^{\Gamma}\mathbb{E}\Big[x^T(\tau)Qx(\tau)+[\mathbb{E}x(\tau)]^T\bar{Q}\mathbb{E}x(\tau)\notag\\
&+\sum_{i=0}^{h}\big[v_{i}^T(\tau)\mathcal{R}_{i}v_{i}(\tau)+[\mathbb{E}v_{i}(\tau)]^T\bar{\mathcal{R}}_{i}\mathbb{E}v_{i}(\tau)\big]\Big]\notag\\
&+\mathbb{E}[x^T(\Gamma+1)\Phi(\Gamma+1)x(\Gamma+1)]+[\mathbb{E}x(\Gamma+1)]^T\bar{\Phi}(\Gamma+1)\mathbb{E}x(\Gamma+1).\notag
\end{align}
Subsequently, by using (\ref{equ:54}), we have
\begin{align} J_{\Gamma}(v)=&\mathbb{E}[x^{T}(0)\theta(-1)]+\sum_{i=0}^{h}\sum_{\tau=0}^{i-1}\mathbb{E}\Big[v_{i}^T(\tau)\mathcal{R}_{i}v_{i}(\tau)+[\mathbb{E}v_{i}(\tau)]^T\bar{\mathcal{R}}_{i}\mathbb{E}v_{i}(\tau)\Big].\notag
\end{align}
Finally, the optimal cost functional is given, with the optimal control $v_{i}(\tau)$ in (\ref{equ:30}), we can obtain that
\begin{align} J_{\Gamma}^{*}(v)=&\mathbb{E}[x^{T}(0)\theta(-1)]+\sum_{i=0}^{h}\sum_{\tau=0}^{i-1}\mathbb{E}\Big[v_{i}^T(\tau)\mathcal{R}_{i}v_{i}(\tau)+[\mathbb{E}v_{i}(\tau)]^T\bar{\mathcal{R}}_{i}\mathbb{E}v_{i}(\tau)\Big]\notag\\
=&\mathbb{E}\Big[x^{T}(0)\Phi(0)x(0)+[\mathbb{E}x(0)]^{T}\bar{\Phi}(0)\mathbb{E}x(0)+x^{T}(0)\sum_{j=1}^{h}\varphi_{j}(0)\hat{x}_{0/-j}\Big]\notag\\
&+\sum_{i=0}^{h}\sum_{\tau=0}^{i-1}\mathbb{E}\Big[v_{i}^T(\tau)\mathcal{R}_{i}v_{i}(\tau)+[\mathbb{E}v_{i}(\tau)]^T\bar{\mathcal{R}}_{i}\mathbb{E}v_{i}(\tau)\Big]\notag\\
=&x^{T}(0)[\Phi(0)+\bar{\Phi}(0)+\sum_{j=1}^{h}\varphi_{j}(0)]x(0)+\sum_{i=0}^{h}\sum_{\tau=0}^{i-1}v_{i}^T(\tau)[\mathcal{R}_{i}+\bar{\mathcal{R}}_{i}]v_{i}(\tau),\notag
\end{align}
which is (\ref{equ:33}). The proof is complete.
\end{proof}

\end{document}